\documentclass[11pt]{amsart}   	
\usepackage{geometry}             
\usepackage[dvipsnames]{xcolor}   		             		
\usepackage{graphicx}				
\usepackage{amssymb,mathrsfs}
\usepackage{amsthm,amsmath,stmaryrd}
\usepackage{tikz}
\usepackage{tikz-cd}
\usepackage{accents,upgreek,enumerate}
\usepackage[headings]{fullpage}
\usepackage{bm}
\usepackage{mathtools}
\usepackage[all]{xy}
\usepackage{caption}
\usepackage[euler-digits]{eulervm}
\usepackage[normalem]{ulem}
\usepackage{colonequals}

\tikzset{
	commutative diagrams/.cd, 
	arrow style=tikz, 
	diagrams={>=stealth}
}

\newtheorem{innercustomthm}{Theorem}
\newenvironment{customthm}[1]
{\renewcommand\theinnercustomthm{#1}\innercustomthm}
{\endinnercustomthm}

\newtheorem{innercustomconj}{Conjecture}
\newenvironment{customconj}[1]
{\renewcommand\theinnercustomconj{#1}\innercustomconj}
{\endinnercustomconj}

\makeatletter
\def\@tocline#1#2#3#4#5#6#7{\relax
	\ifnum #1>\c@tocdepth % then omit
	\else
	\par \addpenalty\@secpenalty\addvspace{#2}%
	\begingroup \hyphenpenalty\@M
	\@ifempty{#4}{%
		\@tempdima\csname r@tocindent\number#1\endcsname\relax
	}{%
		\@tempdima#4\relax
	}%
	\parindent\z@ \leftskip#3\relax \advance\leftskip\@tempdima\relax
	\rightskip\@pnumwidth plus4em \parfillskip-\@pnumwidth
	#5\leavevmode\hskip-\@tempdima
	\ifcase #1
	\or\or \hskip 1em \or \hskip 2em \else \hskip 3em \fi%
	#6\nobreak\relax
	\dotfill\hbox to\@pnumwidth{\@tocpagenum{#7}}\par
	\nobreak
	\endgroup
	\fi}
\makeatother

\usetikzlibrary{calc}
\usetikzlibrary{fadings}
\usetikzlibrary{decorations.pathmorphing}
\usetikzlibrary{decorations.pathreplacing}

\newcounter{marginnote}
\DeclareMathAlphabet{\mathpzc}{OT1}{pzc}{m}{it}

\usepackage[backref=page]{hyperref}
\hypersetup{
	colorlinks   = true,          %Colors links instead of ugly boxes
	urlcolor     = violet,          %Color for external hyperlinks
	linkcolor    = blue,          %Color of internal links
	citecolor   = violet             %Color of citations
}

\newtheorem{theorem}{Theorem}[subsection]
\newtheorem{corollary}[theorem]{Corollary}
\newtheorem{lemma}[theorem]{Lemma}
\newtheorem{proposition}[theorem]{Proposition}

\newtheorem{quasi-theorem}[theorem]{Quasi-Theorem}

\theoremstyle{definition}
\newtheorem{definition}[theorem]{Definition}
\newtheorem{warning}[theorem]{$\lightning$ Warning}
\newtheorem{remark}[theorem]{Remark}

\newtheorem{blank remark}[theorem]{}

\newtheorem{not1}[theorem]{Notation}

\newcommand{\CC} {{\mathbb C}}

\newcommand{\PP}{\mathbb{P}}         
		
\newcommand{\RR} {{\mathbb R}}		
\newcommand{\ZZ} {{\mathbb Z}}		

\newcommand{\cal}{\mathcal}

\def\cM{{\cal M}}

\def\fM{\mathfrak{M}}

\newcommand{\Mbar}{\overline{\cM}\vphantom{\cM}}

\def\trop{\mathsf{trop}}

\newcommand{\Spec}{\operatorname{Spec}}

\makeatletter
\def\blfootnote{\xdef\@thefnmark{}\@footnotetext}
\makeatother

\title{Logarithmic Gromov--Witten theory and double ramification cycles}

\date{}

\author{Dhruv Ranganathan {\it \&} Ajith Urundolil Kumaran}

\address{Dhruv Ranganathan \\ Department of Pure Mathematics {\it \&} Mathematical Statistics\\
	University of Cambridge, Cambridge, UK}
\email{\href{mailto:dr508@cam.ac.uk}{dr508@cam.ac.uk}}

\address{Ajith Urundolil Kumaran \\ Department of Pure Mathematics {\it \&} Mathematical Statistics\\
	University of Cambridge, Cambridge, UK}
\email{\href{mailto:au270@cam.ac.uk}{au270@cam.ac.uk}}

\begin{document}

	\maketitle

	\begin{abstract}
		We examine the logarithmic Gromov--Witten cycles of a toric variety relative to its full toric boundary. The cycles are expressed as products of double ramification cycles and natural tautological classes in the logarithmic Chow ring of the moduli space of curves. We introduce a simple new technique that relates the Gromov--Witten cycles of rigid and rubber geometries; the technique is based on a study of maps to the logarithmic algebraic torus. By combining this with recent work on logarithmic double ramification cycles, we deduce that all logarithmic Gromov--Witten pushforwards, for maps to a toric variety relative to its full toric boundary, lie in the tautological ring of the moduli space of curves. A feature of the approach is that it avoids the as yet undeveloped logarithmic virtual localization formula, instead relying directly on piecewise polynomial functions to capture the structure that would be provided by such a formula. The results give a common generalization of work of Faber--Pandharipande, and more recent work of Holmes--Schwarz and Molcho--Ranganathan. The proof passes through general structure results on the space of stable maps to the logarithmic algebraic torus, which may be of independent interest. 
	\end{abstract}
	
	\vspace{0.3in}
	%\eject
	\setcounter{tocdepth}{1}
	\tableofcontents
	
	\section*{Introduction}
	
	The logarithmic Chow theory of the moduli space of curves $\Mbar_{g,n}$ concerns the cycle theory of the system of birational models of $\Mbar_{g,n}$ obtained by blowups along boundary strata, considered simultaneously. It has attracted significant recent interest~\cite{HMPPS,HPS19,HS21,MPS21,MR21}. Logarithmic Gromov--Witten theory is a basic source of classes in this ring, parallel to the manner in which Gromov--Witten theory is a source of ordinary Chow classes in $\Mbar_{g,n}$. The simplest of these classes arise when the target is a pair $(X,D)$ where $X$ is a smooth projective toric variety and $D$ is the full toric boundary. The purpose of this article is to explain how to relate these logarithmic Gromov--Witten cycles to the logarithmic double ramification cycle, and deduce that the cycles lie in the tautological part of the logarithmic Chow ring of $\Mbar_{g,n}$.
	
	\subsection{Logarithmic Chow rings} We recall the {\it logarithmic Chow ring} of $\Mbar_{g,n}$. Given a stack $B$ with a normal crossings divisor $E$, a {\it simple blowup} is the blowup of $B$ at a smooth stratum. The blowup has a normal crossings divisor given by the  reduced preimage of $E$. A {\it simple blowup sequence} is obtained by successive simple blowups. If $B'\to B$ is a simple blowup sequence, pullback gives rise to an injective homomorphism on Chow rings. The {\it logarithmic Chow ring} is the colimit of Chow rings over all such blowups. The logarithmic Chow ring of $\Mbar_{g,n}$ with respect to the divisor of singular curves is denoted $\mathsf{logCH}^\star(\Mbar_{g,n})$. See~\cite[Section~3]{MPS21} for an introduction. There are two simple ways to produce classes in the logarithmic Chow ring of $B$. 
	
	\begin{enumerate}[]
		\item {\it (Homological)}. Let $Z$ be another simple normal crossings pair and $Z\to B$ a proper morphism of pairs. We produce a class in the Chow ring of $X$ by pushing forward the class $[Z]$. If $B'\to B$ is a simple blowup sequence, then the strict transform of $Z$ gives a class in the Chow ring of $B'$ by the same construction. The class on $B'$ and $B$ are related by pushforward along $B'\to B$, but are {not necessarily} related by pullback. However, for sufficiently fine blowups, the classes stabilize under pullback and so define a class in the logarithmic Chow ring of $B$. See~\cite[Sections~2,3]{MR21}. \smallskip
		\item {\it (Cohomological)} Associated to $B$ is a generalized cone complex $\Sigma(B)$. To each subdivision of $\Sigma(B)$ we associate the ring of {\it strict piecewise polynomials}. This is a very concrete ring: if $\Sigma(B)$ is a simplicial fan, it is the face algebra of the fan. The ring of piecewise polynomials $\mathsf{PP}(\Sigma(B))$ is the colimit of the rings obtained from all subdivisions. It has a ring homomorphism to the logarithmic Chow ring of $B$. The classes can roughly be thought of as polynomials in the Chern classes of bundles built from the strata of $B$ and its blowups. See~\cite{HS21,MPS21,MR21} for details. 
	\end{enumerate}

	\subsection{Logarithmic Gromov--Witten cycles} Let $\mathsf M_\Uplambda(X)$ be the moduli space of logarithmic stable maps to the pair $(X,D)$ where $X$ is a projective toric variety and $D$ is the full toric boundary. The discrete data $\Uplambda$ records the genus $g$, number of markings $n$, and their contact orders with the toric boundary~\cite{AC11,Che10,GS13}. The moduli space is equipped with a basic diagram:
	\[
	\begin{tikzcd}
	\mathsf M_\Uplambda(X)\arrow{d}{\pi}\arrow{r}{\mathsf{ev}}& \mathsf{Ev}_\Uplambda(X)\\
	\Mbar_{g,n}.& 
	\end{tikzcd}
	\]
	The evaluation space $\mathsf{Ev}_\Uplambda$ is constructed as follows. For a marked point $p_i$ and a divisor $D_j$ there is an associated non-negative contact order $c_{ij}$. For a fixed point $p_i$, the intersection of divisors with which $p_i$ has positive contact order is a stratum of $X$. If the contact order is $0$ with all divisors the stratum is defined to be $X$. The evaluation space is the product of these strata associated taken over all marked points. 
	
	A \textit{primary logarithmic Gromov--Witten cycle} is a class of the form
	\[
	\pi_\star\left(\mathsf{ev}^\star(\gamma)\cap [\mathsf M_\Uplambda(X)]^{\mathsf{vir}}\right) \ \text{in  } \mathsf{logCH}^\star(\Mbar_{g,n}).
	\]
	The lift to logarithmic Chow is explained in Section~\ref{sec: log-cycle}. See also the forthcoming survey~\cite{HMPW}. 
	
	We would like to use the notion of a {\it tautological subring} of $\mathsf{logCH}^\star(\Mbar_{g,n})$. A few different definitions of this have been proposed, but all include: (i) all tautological classes pulled back from $\mathsf{CH}^\star(\Mbar_{g,n})$, and (ii) classes from piecewise polynomials. A class is {\it tautological} in logarithmic Chow if it is a polynomial combination of these. 
	
	\begin{customthm}{A}
		Primary logarithmic Gromov--Witten cycles of toric pairs $(X,D)$ lie in the tautological subring of the logarithmic Chow ring $\mathsf{logCH}^\star(\Mbar_{g,n})$. In particular, after pushing forward to the standard Chow ring $\mathsf{CH}^\star(\Mbar_{g,n})$, the cycles lie in the tautological subring. 
	\end{customthm}
	
	Note that descendants at marked points with trivial contact order are pulled back from $\Mbar_{g,n}$, see~\cite[Section~3]{MR16}. By the projection formula, these can be included in the theorem. Other descendants are not pulled back, but the correction terms are boundary strata on the moduli space of maps, so we expect they can be treated inductively. We do not attempt this as we feel it would obfuscate the geometry in our arguments. 
	
	The above result is a parameterized version of the results of Holmes--Schwarz and Molcho--Ranganathan~\cite{HS21,MR21}. When $X$ is $\mathbb P^1$, Faber and Pandharipande proved this statement in their study of relative maps and tautological classes~\cite{FP}. 
	
	\subsection{Enumerative invariants via piecewise polynomials} The theorem above is a corollary of the following stronger result, which relates logarithmic Gromov--Witten cycles for toric pairs to logarithmic double ramification cycles on the moduli space of curves, as studied in~\cite{HMPPS,HS21,MR21}. 
	
	The theorem consists of two pieces. First, in the main text, we associate to every cohomology class $\gamma$ on $\mathsf{Ev}_\Uplambda(X)$ a logarithmic cohomology class $\gamma_{\mathsf{rub}}$ on $\Mbar_{g,n}$. It is given by a piecewise polynomial on the tropical moduli space of curves~\cite{MPS21} and can be calculated by a combinatorial procedure. Second, the discrete data $\Uplambda$ determines a class in the logarithmic Chow ring of $\Mbar_{g,n}$ called the {\it toric contact cycle} or {\it higher double ramification cycle}. It is tautological in the sense above, and on restriction to the moduli of smooth curves, it is the locus of curves that admit a map to a toric variety with discrete data $\Uplambda$, see~\cite{HS21,MR21}. We denote it $\mathsf{TC}_g(\Uplambda)$. 
	
	\begin{customthm}{B}\label{thm:B}
		There is an equality of classes
		\[
		\pi_\star\left(\mathsf{ev}^\star(\gamma)\cap [\mathsf M_\Uplambda(X,D)]^{\mathsf{vir}}\right)  = \gamma_{\mathsf{rub}}\cap \mathsf{TC}_g(\Uplambda), \ \ \textrm{in     } \mathsf{logCH}^\star(\Mbar_{g,n})
		\]
		where $\mathsf{TC}_g(\Uplambda)$ is the logarithmic toric contact cycle associated to the data $\Uplambda$. 
	\end{customthm}
	
	The result is a replacement, in the setting of logarithmic maps, for virtual localization on maps to projective space, proved by Graber--Pandharipande~\cite[Section~4]{GP99}. Parallel to that result, this one gives a complete in-principle solution to the logarithmic Gromov--Witten theory of toric varieties, relative to their full toric boundary.% It is, to our knowledge, the first method to do so. 
	
	The class $\mathsf{TC}_g(\Uplambda)$ is a product of logarithmic double ramification cycles in the ring $\mathsf{logCH}^\star(\Mbar_{g,n})$. A formula for the cycles has recently been established~\cite{HMPPS} and together with the \texttt{admCycles} package~\cite{DSvZ} this gives a practical route to calculations and reduces the computation of logarithmic Gromov--Witten invariants for toric varieties to tautological integrals. Although the complexity of the procedure is significant, this is the first known procedure that completely determines the logarithmic Gromov--Witten theory of a toric variety from known calculations. 
	
	One can approach the logarithmic Gromov--Witten theory of toric varieties via many techniques, several based on tropical curves. See~\cite{Bou19,CJMR2,Gro14,Gro15,MR16,NS06,Par17,Wu21} and the references therein. Via tropical geometry, the result above can be seen to have concrete consequences. For example, it can be used to calculate the Severi degrees of $\mathbb P^2$ and the Hurwitz numbers of $\mathbb P^1$, see Section~\ref{sec: severi-degrees}.
	
	The calculation of the class $\gamma_{\mathsf{rub}}$ is explained in the proof of the theorem, but we summarize it here. In Section~\ref{sec: evaluation-spaces} we construct a toric variety $\mathsf{Ev}_\Uplambda^{\mathsf{rub}}(X)$ together with an equivariant, flat, and proper morphism $\mathsf{Ev}_\Uplambda(X)\to \mathsf{Ev}_\Uplambda^{\mathsf{rub}}(X)$. Given a class $\gamma$ on the evaluation space, we can push it forward and then lift it to an equivariant class on $\mathsf{Ev}_\Uplambda^{\mathsf{rub}}(X)$. Interpreting this as a class $\gamma'$ on the Artin fan, we argue that the morphism from $\mathsf{TC}_g(\Uplambda)$ to this Artin fan factors through a blowup of the moduli space of curves. The pullback of $\gamma'$ to this blowup is a piecewise polynomial $\gamma_{\mathsf{rub}}$. The equality in the theorem is proved at the end of the final section.
	
	\subsection{Conjectures}\label{sec: conjectures} The ordinary Gromov--Witten cycles of any toric variety lie in the tautological ring of the moduli space of curves by the virtual localization formula. This is perhaps the first important consequence of the formula~\cite{GP99}. A parallel formula in logarithmic Gromov--Witten theory has not yet materialized, although progress has been made by Graber~\cite{Gra19}. Even if such a formula appears, it seems likely that its complexity will be similar to that of the degeneration formula~\cite[Theorem~B]{R19}. The result presented here, by comparison, is rather more direct, though it only handles the ``full toric boundary case''. We record the following statements as conjectures, so they might attract some interest. We begin with a weak version.
	
	\begin{customconj}{C}
		Let $X$ be a smooth projective toric variety and let $D$ be a subset of the toric boundary. Logarithmic Gromov--Witten cycles lie in the tautological subring of the logarithmic Chow ring $\mathsf{logCH}^\star(\Mbar_{g,n})$. 
	\end{customconj}
	
	The conjecture above would likely be a consequence of a logarithmic virtual localization formalism. The next, stronger conjecture, is unlikely to fall to localization alone.
	
	\begin{customconj}{D}
		Let $X$ be a smooth projective toric variety and $D$ a simple normal crossings divisor. Logarithmic Gromov--Witten cycles lie in the tautological subring of the logarithmic Chow ring $\mathsf{logCH}^\star(\Mbar_{g,n})$. 
	\end{customconj}
	
	The case where $X$ is $\mathbb P^r$ and $D$ is a normal crossings union of hyperplanes is already very interesting, due to the interplay with the theory of matroids~\cite{RU22}. If $X$ is $\PP^1$ the conjecture holds by work of Faber--Pandharipande~\cite{FP}. In addition to being stronger, it is slightly less clear that the result should hold. One might believe it because of {\it descendant reconstruction}, i.e. the expectation that numerical logarithmic Gromov--Witten theory is determined to leading order by the numerical ordinary Gromov--Witten theory with descendants. This is a theorem in the smooth pair case~\cite{MP06}. If the transfer to the ordinary theory can be made and the subleading terms controlled, virtual localization would yield the result. Orbifold methods may provide another route to the use of torus actions, see for instance~\cite{BNR22}. 
	
	The point of our approach here is to avoid localization, but torus actions do lurk in the background. Torus localization is still a crucial part of the study of the double ramification cycle~\cite{HMPPS,JPPZ}. Also, in the proof of main results, classes from the evaluation space are replaced by torus equivariant lifts, given by piecewise polynomials on the cone complex of the evaluation space. 
	
	\subsection{Overview} The high level strategy in the paper is simple: apply the Abramovich--Wise birational invariance theorem to replace the toric varieties by products of $\mathbb P^1$'s, then apply the product formula in logarithmic Gromov--Witten theory, and finally reduce to the double ramification cycle via the rubber geometry~\cite{AW,Herr,HMPPS,JPPZ,MR21,R19b}. However, as stated, the strategy is nonsensical.%, as we now explain. 
	
	\subsubsection{Gaps in the naive strategy} The first issue is that birational invariance only relates the virtual classes of curves in $X$ and in $(\mathbb P^1)^r$ by pushforward, but an arbitrary insertion is unlikely to be pulled back along $X\to(\mathbb P^1)^r$ so this doesn't help; this was raised by Abramovich--Wise~\cite[Section~1.4]{AW} as a missing piece in their result. We realize that if one works in the logarithmic Chow ring, the missing insertions become accessible. The relevant strengthening of the Abramovich--Wise result was accidentally recorded in~\cite[Section~3]{R19}.  
	
	The second issue, which is closely related, is that the simplest form of the product formula doesn't hold in logarithmic Gromov--Witten theory. However, it does hold if one works over a blowup of the moduli space of curves, or equivalently in the logarithmic Chow ring~\cite{Herr,R19b}. 
	
	The third issue, which is of a different nature, is that the double ramification cycle deals with the {\it rubber} geometry rather than the rigid geometry, so we need to relate the virtual classes of the rubber and rigid mapping spaces. And even if this is done, the insertions need to be accessed on the rubber geometry, and evaluation spaces for rubber have not previously been considered. 
	
	%The approach to control the logarithmic Gromov--Witten cycles taken here is to rely on the fact that the cycle of curves in $\Mbar_{g,n}$ parameterizing curves that admit a map of type $\Uplambda$ to a toric variety is now well-understood, via product formulas and logarithmic double ramification cycles~\cite{HMPPS,HPS19,HS21,MR21}. These cycles control the ``rubber geometry'' of $X$, which is roughly the quotient of the space of logarithmic maps to $X$ by the natural dilating torus. The problem then lies in transferring integrals against the virtual class for the parameterized target to the rubber target. The solution to this problem is the main technical contribution of the paper. 
	
	\subsubsection{Torus quotient problems} Another key geometric idea in the paper is the {\it rubber evaluation space} and associated evaluation maps from the space of rubber maps. These have a simple relationship with the corresponding rigid evaluation spaces and maps. 
	
	The basic idea is as follows: just as the rubber moduli space parameterizes maps from curves up to the action of the dense torus $T$ of $X$, the rubber evaluation space should parameterize configurations of points on strata up to the action of $T$. The elementary but crucial observation is that $\mathsf{Ev}_\Uplambda(X)$ should be considered as a single object with the diagonal action of the torus $T$, via the natural action of $T$ on its strata. We are guided here by observations of Carocci--Nabijou~\cite{CN21}. 
	
	A geometric complication in this idea is that while we have spoken about the rubber moduli space of stable maps above, the only sensible meaning for such a space is via an intersection problem in the Picard stack, using~\cite{MR21} and building on~\cite{Hol17,MW17}. This needs to be related to logarithmic mapping spaces. The mapping space with a toric target appears, on the interior, to be a torus bundle over the rubber. However, the extension of the description to the boundary is not clear, even in rank $1$. The construction of our proposed rubber evaluation spaces presents the same issues: it is straightforward on the interior, where the torus action is free, but complicated in the boundary. 
	
	In principle, the resulting quotient problems are as complicated as the problem of constructing quotients of toric variety by subtori. Even if one assumes a solution to this, the global structure of the map from the space of rigid maps to a toric variety to the space of rubber maps is complicated. 
	
	\subsubsection{Logarithmic algebraic tori} While these quotient problems are subtle, one might expect that by birational invariance results in logarithmic geometry, the combinatorial choices made in the toric quotient problem above are immaterial~\cite{AW,R19}. In order to use this flexibility, we work with the space of stable maps to the {\it logarithmic algebraic torus}. This is a non-representable group-valued functor on logarithmic schemes that is, in some sense, birational to the algebraic torus, see~\cite{RW19}. Although not representable, it is compact, and therefore quotient constructions by this group are straightforward. Stable maps to the logarithmic torus bundles arise implicitly throughout logarithmic Gromov--Witten theory, via the expansions in~\cite{R19}. We believe they will find use wherever immaterial non-canonical polyhedral choices need to be made. 
	
	Allowing logarithmic tori, there are essentially no difficulties in setting up the quotient problems. We easily construct the space of rubber maps, the rubber evaluation space, and compare with the rigid geometry. Once everything is constructed, birational modifications can be made to make things representable. We then move the problem into the logarithmic Chow ring of $\Mbar_{g,n}$. 
	
	\subsection{Broader context} A question motivating this work is whether Gromov--Witten cycles always lie in the tautological singular cohomology of the moduli space of curves; this is a central conjecture in the subject going back to work of Levine--Pandharipande~\cite[Section~0.8]{LP09}. The reader may refer to~\cite{ABPZ,Jan17,MP06} for a sampling of results. The statement is interesting for several reasons. For one, the tautological ring is a very small part of the cohomology ring, so the constraint is strong. Second, the tautological ring is highly structured: for example, it has an additive set of generators with a conjecturally complete set of relations~\cite{PPZ15}. And finally, it is sufficiently well-understood that it can be explored using computer algebra methods~\cite{DSvZ}. 
	
	A promising approach to this conjecture is via the logarithmic degeneration methods in~\cite{R19} and the study of vanishing cycles in~\cite{ABPZ}. The crucial inputs into this are the study of toric varieties and toric fibrations over logarithmic targets. We initiate such a study for toric varieties with respect to the full toric boundary, and will further develop these methods to study toric fibrations in future work. The conjectures in Section~\ref{sec: conjectures} are central to this question.
	
	The results here fit into a broader study of the logarithmic tautological ring, which has seen a flurry of recent activity~\cite{HMPPS,HPS19,HS21,MPS21,MR21}. The present paper illustrates the utility of understanding the cycle theory of this ring, showing for example that the multiplicative structure of logarithmic tautological classes recovers the enumerative geometry of all toric pairs. It is parallel to the recovery of absolute Gromov--Witten invariants by Hodge integrals. 
	
	While the present paper gives a complete in-principle solution to the logarithmic Gromov--Witten theory of toric varieties, it an interesting problem to make the result explicit for various classes of invariants. For example, double Hurwitz numbers are explored in~\cite{CMR22}. More recently, Kennedy-Hunt, Shafi, and the second author use the perspective in the present paper to give explicit formulas for the stationary descendant Gromov--Witten theory of $S\times\mathbb A^1$, where $S$ is a toric surface relative to its toric boundary~\cite{KHSUK}, recovering the formulas of Bousseau~\cite{Bou19}. %The method is to apply tautological relations on the logarithmic double ramification cycle to the formulas appearing here. 
	
	%In ongoing work, Kennedy-Hunt, Shafi, and the second author produce explicit logarithmic Chow cohomology classes that extract the toric Severi degrees of toric surfaces, and enumerative invariants in higher dimension, using these ideas. 
	
	A more speculative question is the link to traditional virtual localization. When the logarithmic boundary is empty, localization describes the invariants via Hodge integrals on the moduli space of curves. When the boundary is toric, our methods here describe the invariants via double ramification integrals. It is likely that future applications will demand that this gap is bridged.
	%It will be important in future applications to bridge this gap. %We propose that a general torus localization setup in logarithmic Gromov--Witten theory may have to interpolate between these two cases. 
	
	\subsection{User's guide} In Section~\ref{sec: log-torus} we recall the logarithmic torus and discuss stable maps into it. The section includes a discussion of stable maps to toric varieties and Artin fans. The main content of the section beyond the construction is the relationship between the rigid and rubber mapping spaces. In Section~\ref{sec: evaluation-spaces} we construct evaluation spaces for stable maps to logarithmic tori, and for its rubber variant. In Section~\ref{sec: main-theorems} we prove the main results. After recalling the construction of logarithmic Gromov--Witten cycles, we compare the rigid/rubber virtual structures, and use this, with the logarithmic intersection yoga of~\cite{MR21}, to deduce the results. 
	
	Throughout the paper, we work with fine and saturated logarithmic schemes and stacks, and with stacks over this category. The logarithmic structures will typically be defined in the \'etale topology, with the exception of several Artin stacks that appear, whose logarithmic structures are defined in the lisse \'etale or fppf sites. Fiber squares will be understood to be in the category of fine and saturated logarithmic schemes (or stacks, as appropriate).
	
	\subsection*{Acknowledgements} We are grateful to D. Abramovich, L. Battistella, D. Holmes, D. Maulik, R. Pandharipande, J. Schmitt, and N. Nabijou for many instructive discussions on rubber geometry and double ramifications over the years. R. Cavalieri, H. Markwig, S. Molcho, and J. Wise are owed special thanks: the collaboration~\cite{CMR22} provided inspiration for the statements proved here, S. Molcho taught us a great deal about logarithmic intersection theory in~\cite{MR21}, and the first author learned about the logarithmic torus from J. Wise in discussions surrounding~\cite{RSW17B,RW19}. Finally, we thank F. Carocci, D. Holmes, R. Pandharipande, and the anonymous referee for very helpful comments on earlier drafts, and S. Molcho and J. Wise for helping us sharpen some proofs in the present version.
	
	\noindent
	D.R. is supported by EPSRC New Investigator grant EP/V051830/1. 
	
	\section{Stable maps to the logarithmic torus}\label{sec: log-torus}
	
	The main constructions in this paper are based on torus quotient problems. Quotient problems for algebraic tori can be subtle, even for quotients of a toric variety by a subtorus of the dense torus~\cite{KSZ91}. The source of the complexity is the non-compactness of the torus. In logarithmic geometry, there is a canonical group compactifying the $\mathbb G_m$, and quotients by it are better behaved. We use this and birational invariance for logarithmic maps~\cite{AW} to prove our main results.
	
	\subsection{Subdivisions and Artin fans} A logarithmic morphism $Y\to X$ with $X$ logarithmically smooth is a {\it subdivision} if it is proper, birational, and logarithmically \'etale. More generally, if $X$ is any logarithmic scheme, $Y\to X$ is a subdivision if, locally on $X$, it is the strict base change of a subdivision of a logarithmically smooth scheme. We do not require the map to be representable by schemes, but only by Deligne--Mumford stacks, so generalized root constructions are permitted; this is what is termed as a {\it logarithmic modification} by Abramovich--Wise~\cite{AW}, but is called a {\it logarithmic alteration} by some authors. The relevant subdivisions for us will be globally expressed as pullbacks of subdivisions of logarithmically smooth stacks. 
	
	Let $\sigma$ be a cone in the sense of toric geometry and let $X(\sigma)$ be the associated affine toric variety with torus $T$. The {\it Artin cone} associated to $\sigma$ is the algebraic stack $\mathsf A_\sigma = [X(\sigma)/T]$. It carries a logarithmic structure in the smooth topology, descended from $X(\sigma)$ and has the property that $\mathsf A_\sigma\to\Spec \mathbb C$ is logarithmically \'etale. 
	
	More generally, one can consider stacks $\mathsf A$ with logarithmic structure such that $\mathsf A\to \Spec \mathbb C$ is logarithmically \'etale. A stack $\mathsf A$ of this form is an {\it Artin fan} if it has a strict \'etale cover by Artin cones. Typical examples are stacks $\mathsf A_X = [X/T]$, for $X$ a toric variety. These stacks will arise as targets in logarithmic mapping stacks. We will only need a few basic facts about Artin fans beyond the definition, and refer the reader to the references~\cite{ACMUW,CCUW} which are both well-adapted to our point of view here. 
	
	\subsection{Setup and numerical data}\label{sec: numerica-data} Let $X$ be a smooth and projective toric variety of dimension $r$. Equip it with its toric logarithmic structure. Let $N$ denote the cocharacter lattice of the torus.
	
	We will consider maps from logarithmic curves to the target $X$. The symbol $\Uplambda$ will denote the numerical data of the moduli problem for such maps. If we fix an isomorphism of $N$ with $\ZZ^r$ this is: a genus $g$, a number $n$ of marked points, and a contact order matrix $\underline A$ with $r$ rows and $n$ columns.\footnote{If we do not choose a basis, the contact order matrix should be thought of as a $n$-tuple of vectors in $N$.} We assume the row sums are equal to $0$; the relevant moduli spaces are empty otherwise. The $j^{\textrm{th}}$ column of $\underline A$ should be thought of as attached to $j^{\textrm{th}}$ marked point. %The $j^{\textrm{th}}$ column vector $v_j$ of $\underline A$ determines a vector in $N$ which serves as the {contact order} for the problem.
	
	We pause to explain the contact order. Let $\Sigma$ be the fan of $X$. Fix an index $j$, with $1\leq j\leq n$. The vector $v_j$ is contained in the relative interior of a unique cone, say $\sigma$. Let $u_{i_1},\ldots, u_{i_p}$ be the primitive generators of $\sigma$. Write
	\[
	v_j = \sum_{k=1}^p a_{ji_k}u_{i_k}, \ \ a_{ji_k}\in\mathbb N_{>0}.
	\]
	Given this, we can restrict attention to maps from pointed curves
	\[
	(C,p_1,\ldots,p_n)\to (X,D),
	\]
	where the divisor $D_{i_k}$ corresponding to the generator $u_{i_k}$ has contact order $a_{ji_k}$ with the marked point $p_j$. In this way, the matrix $\underline A$ specifies the contact orders. More generally, we can consider logarithmic maps with these contact data. 
	
	\subsection{Maps to toric varieties and Artin fans}\label{sec: maps-to-tvs} We consider moduli $\mathsf M_\Uplambda(X)$ of logarithmic stable maps to the toric variety $X$ with its toric logarithmic structure. The moduli functor is representable by a proper Deligne--Mumford stack with a logarithmic structure and a virtual class~\cite{AC11,Che10,GS13}. 
	
	Let $\mathsf A_X$ be the stack $[X/T]$ equipped with the logarithmic structure descended from $X$. It is called the {\it Artin fan of $X$}. One can consider moduli of logarithmic maps from {\it prestable curves} to $\mathsf A_X$ with numerical data $\Uplambda$. The spaces are denoted $\mathfrak M_\Uplambda(\mathsf A_X)$ and they are Artin stacks with logarithmic structure, see~\cite{AW}. The natural morphism $\mathsf M_\Uplambda(X)\to \mathfrak M_\Uplambda(\mathsf A_X)$ is strict.
	
	We impose a ``stability condition'' on $\mathfrak M_\Uplambda(\mathsf A_X)$. The idea is that the condition that a prestable map $C\to X$ is stable is visible after projecting to $\mathsf A_X$ -- this is a toric phenomenon. Indeed, stability needs to be examined on rational components with one or two special points. A $1$-pointed component is automatically contracted, for example by the balancing condition. A $2$-pointed component is contracted if and only if the two special points both map to the same locally closed toric stratum.\footnote{The discussion and pictures in~\cite[Section~3]{R15b} may be helpful for the reader.} Thus, the conditions depend only on the dual graph of $C$ and the strata to which the special points map. We can therefore impose the conditions on logarithmic maps $C\to \mathsf A_X$. 
	
	We now do this formally. A prestable logarithmic map
	\[
	C\to\mathsf A_X
	\]
	over a logarithmic point has a combinatorial type~\cite[Section~1]{GS13}. It includes (i) a dual graph $\Gamma$, (ii) a cone of $\Sigma_X$ for each vertex and edge of $\Gamma$, (iii) a primitive direction for each vertex-edge flag, and (iv) a nonnegative weight associated to each such direction. The directions are primitive elements in the lattice $N$.

	\begin{definition}[Linear, balanced, and stable]\label{def: artin-fan-stability}
		Let $C\to \mathsf A_X$ be a prestable map. A vertex of the dual graph $\Gamma$ is {\it linear bivalent} if it has genus $0$, is bivalent, and the slopes of the two flags based at the vertex are opposite, i.e. the edge lies on a line. A linear bivalent vertex is {\it stable} if the cones assigned to the vertex and its two flags do not coincide. 
		
		A vertex of $\Gamma$ is {\it contracted} if the slopes of all flags based at this vertex are $0$. A contracted vertex is {\it stable} if it either has positive genus, or it has genus $0$ and at least $3$ incident flags.  
		
		A prestable map $C\to \mathsf A_X$ is {\it balanced} if at every vertex of $\Gamma$, the weighted sum of the outgoing edge directions at that vertex is $0$. 
		
		A prestable map $C\to \mathsf A_X$ is {\it stable} if it is balanced and all linear bivalent vertices and contracted vertices are stable. 
	\end{definition}
	
	The balancing condition and the stability condition are both open in moduli. Both conditions can be checked at the level of dual graphs, so it suffices to check that the balancing and stability conditions are preserved when specializing the combinatorial type. These can both be checked explicitly. The reader can find the verifications, although unfortunately with slightly different terminology, in~\cite[Section~2.2]{R16}.
	
	\begin{definition}\label{def: stability-artin-fans}
		The moduli stack of {\it stable maps} to the Artin fan $\mathsf A_X$ is the open substack of $\mathfrak M_\Uplambda(\mathsf A_X)$ parameterizing maps that are stable in the sense above. It is denoted $\mathsf M_\Uplambda(\mathsf A_X)$.
	\end{definition}
	
	In short, the space $\mathsf M_\Uplambda(\mathsf A_X)$ are those maps to the Artin fan whose combinatorial type {\it could} come from a logarithmic {\it stable} map to a toric variety. The balancing condition on $C\to\mathsf A_X$ is a necessary condition for the existence of a lift to $X$, though it is typically far from sufficient. 
	
	\begin{warning}\label{warning: stability}
		The space of stable maps to the Artin fan is of finite type, but it fails to be universally closed and is typically an Artin stack. A clearer picture of its properties can be deduced from its description as a subdivision of the functor of maps to $\mathbb G^r_{\trop}$, which is developed in later sections. However, we do not use its geometric properties in a serious way here, so we leave the details to an interested reader. 
	\end{warning}
	
	By the definition of stability, we have a natural morphism
	\[
	\mathsf M_\Uplambda(X)\to\mathsf M_\Uplambda(\mathsf A_X)\subset \mathfrak M_\Uplambda(\mathsf A_X). 
	\]
	The latter two are logarithmically smooth of dimension $3g-3+n$; the space on the left is typically singular. The morphism is strict and equipped with a relative perfect obstruction theory~\cite{AW}. 
	
	\subsection{Logarithmic tori} We introduce the protagonist of our story.
	
	\begin{definition}
		The {\it logarithmic algebraic torus} $\mathbb{G}_{\mathsf{log}}$ is the functor 
		$$
		\textbf{LogSch}\to \textbf{AbGrp}
		$$ 
		whose value on a logarithmic scheme $S$ is
		\begin{equation*}
		\mathbb{G}_{\mathsf{log}}(S)=H^0(S,M^{\mathsf{gp}}_S),
		\end{equation*}
		The {\it tropical torus} $\mathbb G_{\mathsf{trop}}$ is the functor defined by
		\begin{equation*}
		\mathbb{G}_{\mathsf{trop}}(S)=H^0(S,\overline M^{\mathsf{gp}}_S).
		\end{equation*}
	\end{definition}
	
	A logarithmic scheme has an associated generalized cone complex or cone stack, and these are useful when considering maps to $\mathbb G_{\trop}$ and $\mathbb G_{\mathsf{log}}$. The generalized cone complex is a colimit of a diagram of cones with face maps between them, see~\cite[Appendix~B]{GS13} or~\cite{ACMUW} for the association, and~\cite{CCUW} for a detailed discussion of cone stacks. The cone stack is similar but retains some higher categorical data about the diagram. Whenever we use this, either the distinction does not arise or either can be used. 
	
	In more detail, a section of $\overline M^{\mathsf{gp}}_S$ is equivalent to combinatorial data. For families of logarithmic curves $C/S$, we record the relevant notions in the next two remarks. 
	
	\begin{remark}[Moduli of tropical curves]\label{rem: tropical-curves}
		In a few places in this section, we will need a part of the theory of the tropicalization for the moduli stack of curves~\cite{CCUW}. The authors construct an object $\cM_{g,n}^\trop$, a certain stack on the category of cone complexes. Let $\sigma$ be a strictly convex cone with dual monoid $S_\sigma$. A tropical curve over $\sigma$, denoted $C_{\trop}$, is a stable dual graph $\Gamma$ of type $(g,n)$ together with a decoration of each edge by a nonzero element of $S_\sigma$. This decoration is the {\it length} of the edge. Dualizing, a tropical curve over $\sigma$ produces a map of cone complexes $C_\trop\to \sigma$, such that the fibers of this map over the relative interior are metric graphs enhancing $\Gamma$. We are mildly abusing notation here, using $C_{\trop}$ both for the decorated graph and the cone complex. The first main result concerns the fibered category $\cM_{g,n}^{\trop}$ over cone complexes whose value on a cone $\sigma$ is the groupoid of tropical curves over $\sigma$. It is proved in~\cite{CCUW} that $\cM_{g,n}^{\trop}$ is representable by a ``cone stack''. The second main result is that there is a canonically associated Artin fan, deoted $a^\star\cM_{g,n}^{\trop}$ and a strict, smooth, and surjective map:
		\[
		\Mbar_{g,n}\to a^\star\cM_{g,n}^{\trop}.
		\]
		It will arise for us in the following fashion. Fix $S$ an {\it atomic} logarithmic scheme~\cite[Section~2]{AW}. It admits a strict map $S\to \mathsf A_\sigma$ to an Artin fan associated to a cone $\sigma$. Then each family of $(g,n)$ stable curves $C/S$ produces a moduli map $\sigma\to \cM_{g,n}^{\trop}$ or equivalently a tropical curve $C_{\trop}$ over $\sigma$.
	\end{remark}
	
	\begin{remark}[Piecewise linear functions]\label{rem: pl-functions} Let us use the tropical curve to describe sections of $\overline M_C^{\sf gp}$. Suppose $C$ is a logarithmic curve over a closed logarithmic point $S$ and let $\Gamma$ be its stable dual graph, with genus and marking decorations. Building on the previous remark, the dual graph $\Gamma$ can be enhanced with an element of $\overline M_S$ at each edge; this monoid takes the place of $S_\sigma$ in the previous remark. An edge $E$ of $\Gamma$ is dual to a node $q_E$ of $C$. The local equation at $q_E$ is $xy = t$, where $t$ is an element of the monoid sheaf $M_S$ that maps to $0$ in $\mathcal O_S$. The image of $t$ in the characteristic monoid $\overline M_S$ is denoted $\overline t$. The decorated graph is denoted $C_{\trop}$, consistent with the discussion above. The element $\overline t$ is the {\it length} of the edge $E$. As above, the decorated graph $C_{\sf trop}$ is equivalent to a cone complex over the dual cone of $\overline M_S$.
		
		Suppose that $\overline\alpha \in \overline M_C^{\sf gp}$.  If $V$ is a vertex of $C_{\sf trop}$ there is a component of $C$ dual to $V$, and on the interior of this component, $\overline M_C^{\sf gp}$ is constant with value $\overline M_S^{\sf gp}$.  Write $\overline\alpha(v)$ for the value of $\overline\alpha$ on the interior of this component.  If $E$ is an edge between $V$ and $W$, near $E$ we can write:
		$$\overline\alpha=\overline\alpha(V) + m \overline x,$$ 
		where $\overline x$ is the image of $x \in M_{C,E}$ in $\overline M_{C,E}$ and $m \in \mathbb Z$.  Restricting to $W$, we find $\overline\alpha(W) = \overline\alpha(V) + m \overline t$.  We may view $\overline\alpha$ as a piecewise linear function on $C_{\sf trop}$ with slope $m$ on the edge $E$, directed from $V$ to $W$. Globally, we obtain a piecewise linear function on the cone complex $C_{\trop}$.
	\end{remark}
	
	The functors are not representable by schemes, or even stacks, with logarithmic structure. We say that they are ``schematically non-representable''. The following is a replacement:
	
	\begin{proposition}
		The scheme $\mathbb{P}^1$ equipped with its toric logarithmic structure is a subdivision of $\mathbb{G}_{\mathsf{log}}$.
	\end{proposition}
	
	\begin{proof}
		See~\cite[Section~1]{RW19}. 
	\end{proof}
	
	Any complete toric variety $X$ completing $\mathbb G_m\otimes N$ arises uniquely as a subdivision $X\to \mathbb G_{\mathsf{log}}\otimes N$. Analogously, there is a subdivision $
	\mathsf A_X\to\mathbb G_{\mathsf{trop}}^r$.
	
	\noindent
	{\bf Notation.} {\it The symbol $N$ will denote the cocharacter lattice of the dense torus of $X$. However, to compactify notation, we fix an identification of $N$ with $\mathbb Z^r$ and denote the logarithmic torus $N\otimes\mathbb G_{\mathsf{log}}$ by $\mathbb G_{\mathsf{log}}^r$.} 
	
	\begin{remark}
		One can visualize $\mathbb{G}_{\mathsf{log}}^r$ as the ``toric variety'' defined by a fan with exactly one non-strictly convex cone $\sigma=\mathbb{R}^r$. It is not representable by a scheme (or algebraic stack) with logarithmic structure. Generally, if ``strict convexity'' in the definition of a cone from toric geometry is dropped, we obtain non-representable functors ``in between'' toric varieties and $\mathbb G_{\mathsf{log}}^r$. 
	\end{remark}
	
	%Taking the remark above seriously, the toric case of the following lemma is natural. The general case holds. 
	%
	%\begin{lemma}
	%\label{lem:morphdes}
	%We assume we have the following diagram:
	%\begin{equation*}
	%\begin{tikzcd}
	%Z \arrow[r, "g"] \arrow[d, "f"]
	%& \mathbb{G}_{\mathsf{log}}^r  \\
	%Q
	%\end{tikzcd},
	%\end{equation*}
	%where $Z$ and $Q$ are algebraic stacks with logarithmic structure, and $f$ is a subdivision. Suppose additionally $f_\star\mathcal O_Z = \mathcal O_Q$. Then there exists a unique morphism $h: G\to \mathbb{G}_{\mathsf{log}}^r$ such that $h \circ f = g$. In particular, the lemma holds when $G$ is logarithmically smooth (over the base). 
	%\end{lemma}
	%
	%\begin{proof}
	%Under the conditions of the lemma, the global sections of the groupified logarithmic structure sheaf is constant~\cite[Section~4]{MW18}. 
	%\end{proof}

	\subsection{Maps to logarithmic tori}\label{sec: maps-to-tori}
	We define moduli spaces of logarithmic maps with target $\mathbb{G}_{\mathsf{log}}^r$. Reflecting the schematic non-representability of the target, the moduli space is also non-representable. 
	
	We have noted that a complete toric variety gives rise to a canonical subdivision $X\to \mathbb{G}_{\mathsf{log}}^r$. 
	In Proposition~\ref{prop:rigidmod}, we show that this induces subdivision of the corresponding moduli spaces; the analogous statement for schematic targets is proved in~\cite{AW}.  %A proof has already appeared in the genus $0$ case in~\cite{RW19}. We record the general case in Proposition~\ref{prop:rigidmod}.  
	
	\subsubsection{Setup and discrete data} Fix discrete data $\Uplambda=(g, \underline{A}=(A_1,\ldots, A_n))$ where $\underline{A}$ is a tuple of integer vectors each of length $r$. In other words, the contact data $r\times n$ matrix. Each row of this matrix records a {\it contact order} for the moduli problem of logarithmic maps. 
	
	In geometric contexts, contact orders keep track of orders of vanishing of boundary divisors of the target along marked points on the curve. Since we work in this non-representable context, we explain how to think about the contact order. A morphism 
	$$
	\overline s = \Spec(\mathbb N\to \mathbb C)\to \mathbb G^r_{\mathsf{log}}
	$$ 
	is the data of $r$ sections of $M^{\mathsf{gp}}_{\overline s}$. Projecting along the natural map $M^{\mathsf{gp}}_{\overline s}\to \overline M^{\mathsf{gp}}_{\overline s}$, we obtain $r$ integers. Given a morphism from a logarithmic scheme $S$, one can collect this tuple of integers for all geometric points. These data are locally constant.
	
	Given a logarithmic morphism to $\mathbb G_{\mathsf{log}}^r$ from a logarithmic curve $C$ over $S$, the marked points are a family of $\Spec(\mathbb N\to \mathbb C)$-points over $S$. We can analogously fix the contact order at each marked point. We can similarly make sense of the contact order for a flag consisting of a component and a node on it. The contact order of $C\to \mathbb G^r_{\mathsf{log}}$ only depends on the map $C\to \mathbb G^r_{\mathsf{trop}}$.
	
	\subsubsection{Geometry of the mapping space} We describe moduli of maps to logarithmic and tropical tori.
	
	\begin{definition}\label{def: glog-def}
		The {\it moduli space of prestable logarithmic maps to $\mathbb{G}_{\mathsf{log}}^r$}, denoted $\mathfrak{M}_{\Uplambda}(\mathbb{G}_{\mathsf{log}}^r)$, is the fibered category over logarithmic schemes whose fiber over a logarithmic scheme $S$ is:
		\begin{equation*}
		\left( \begin{tikzcd}
		C \arrow[d,swap,"\pi"] \\
		S
		\end{tikzcd}, \ (s_1, \ldots, s_r) \right),
		\end{equation*}
		where $\pi$ is a log curve $C \to S$ of type $(g,n)$ and $s_1, \ldots, s_r$ in $H^0(C, M^{\mathsf{gp}}_C)$ with contact profile $\underline{A}$. 
		
		A morphism between two such pairs is a cartesian diagram:
		\begin{equation*}
		\begin{tikzcd}
		C_1 \arrow[r] \arrow[d] \arrow[dr, phantom, "\square"] & C_2 \arrow[d]\\
		S_1 \arrow[r] & S_2,
		\end{tikzcd}
		\end{equation*}
		such that the two $r$-tuples of sections are related by pullback. We have a natural map:
		\begin{equation*}
		\mathfrak{M}_{\Uplambda}(\mathbb{G}_{\mathsf{log}}^r)\to\mathfrak{M}^{\mathsf{log}}_{g,n},
		\end{equation*}
		where $\mathfrak{M}^{\mathsf{log}}_{g,n}$ is the moduli space of genus $g$ log curves with $n$ markings.  The moduli space of {\it stable logarithmic maps to $\mathbb{G}_{\mathsf{log}}^r$}, denoted  $\mathsf{M}_{\Uplambda}(\mathbb{G}_{\mathsf{log}}^r)$, is the subcategory of $\mathfrak{M}_{\Uplambda}(\mathbb{G}_{\mathsf{log}}^r)$ where the source curves are stable.
		
		In parallel, we define logarithmic maps to $\mathbb G_\trop^r$ as $r$ sections of $\overline M_C^{\mathsf{gp}}$. A map is {\it balanced} at a vertex $V$ if the sum of the contact orders of flags leaving $V$ vanishes, see Remark~\ref{rem: pl-functions}. A map $C\to \mathbb G_\trop^r$ is called {\it stable} if the underlying curve is stable and the map is balanced at all vertices. The moduli of stable maps to $\mathbb G_{\mathsf{trop}}^r$ is denoted $\mathsf M_\Uplambda(\mathbb{G}_{\mathsf{trop}}^r)$. 
	\end{definition}
	
	We come to the relationship between the spaces of maps to logarithmic tori and to tropical tori. If $C\to \mathbb G_{\mathsf{trop}}$ factors through $\mathbb G_{\mathsf{log}}$, then it is automatically balanced, so there is a morphism
	\[
	\mathsf M_\Uplambda(\mathbb{G}_{\mathsf{log}}^r)\to \mathsf M_\Uplambda(\mathbb{G}_{\mathsf{trop}}^r).
	\]
	Although neither $\mathsf M_\Uplambda(\mathbb{G}_{\mathsf{log}}^r)$ nor $\mathsf M_\Uplambda(\mathbb{G}_{\mathsf{trop}}^r)$ is schematically representable, we have the following:
	
	\begin{proposition}\label{prop: strictness}
		The morphism
		\[
		\mathsf M_\Uplambda(\mathbb{G}_{\mathsf{log}}^r)\to \mathsf M_\Uplambda(\mathbb{G}_{\mathsf{trop}}^r)
		\]
		is representable and strict. 
	\end{proposition}
	
	\begin{proof}
		Let $S$ be a logarithmic scheme and fix $S\to \mathsf M_\Uplambda(\mathbb{G}_{\mathsf{trop}}^r)$. We show that the base change map
		\[
		S':=S\times_{\mathsf M_\Uplambda(\mathbb{G}_{\mathsf{trop}}^r)} \mathsf M_\Uplambda(\mathbb{G}_{\mathsf{log}}^r)\to S
		\]
		is strict with schematic domain. Let us describe $S'$. The map $S\to \mathsf M_\Uplambda(\mathbb{G}_{\mathsf{trop}}^r)$ is equivalent to the data of a family of logarithmic curves $C/S$ and a map $\alpha: C\to \mathbb G_{\mathsf{trop}}$. The torsor $\mathbb G_{\mathsf{log}}\to \mathbb G_{\trop}$ pulls back to a $\mathbb G_m$-torsor on $C$, whose total space is denoted $\mathbb G(\alpha)\to C$. Liftings of
		\[
		\begin{tikzcd}
		&{{\mathbb G}_{\mathsf{log}}}\arrow{d}\\
		C\arrow{r} \arrow[dashed]{ur}&{\mathbb G}_{\mathsf{log}}
		\end{tikzcd}
		\]
		are equivalent to the data of sections of $\mathbb G(\alpha)\to C$. The space $S'$ is therefore precisely the moduli space of sections of $\mathbb G(\alpha)\to C$. Since the map $\mathbb G_{\mathsf{log}}\to \mathbb G_{\trop}$ is strict, the map $\mathbb G(\alpha)\to C$ is too. The space $S'$ of sections of the torsor is therefore representable and strict over $S$.  
	\end{proof}
	
	%Suppose X ? Y is a strict map.  For M(X) ? M(Y) to be strict means that
	%whenever we have a log. map C ? Y and a schematic lift f : C ? X, there
	%is a unique promotion of f to a log. map.  But this is clear, since the
	%only datum we would need to supply is f^* M_X ? M_C, which we get
	%automatically because M_X is pulled back from M_Y.
	%
	%In your case, the above doesn't literally make sense because X and Y do
	%not have underlying schemes.  However, the fiber product X ×_Y C is
	%representable, and lifts of the map C ? Y to C ? X correspond to
	%sections of X ×_Y C over X.  If the map C ? Y corresponds to a section ?
	%of?M_C^gp then X ×_Y C is the total space of the ??*-torsor ??(?) (or
	%maybe ??(-?), depending on your point of view --- it doesn't matter).
	%Giving a section of this is exactly equivalent to giving a log. map C ?
	%G_log.
	
	\begin{warning}
		Although we call it a stability condition, the space $\mathsf M_\Uplambda(\mathbb{G}_{\mathsf{trop}}^r)$ should not be thought of as being proper. Since the space is schematically non-representable though, only maps to it from {\it logarithmic} schemes are defined, rather than maps from ordinary schemes. Therefore it is slightly awkward to precisely define properness. One manifestation of the non-properness is that $\mathsf M_\Uplambda(\mathbb{G}_{\mathsf{trop}}^r)$ fails the logarithmic valuative criterion for properness~\cite[Theorem~3.5.2]{MW17}. In this criterion, we start with the spectrum $S$ of a discrete valuation ring, and let $\eta$ be its generic point. We then equip $\eta$ with a logarithmic structure consider maps from it to $\mathsf M_\Uplambda(\mathbb{G}_{\mathsf{trop}}^r)$ and ask when such maps can be extended. In this setup, there exist maps from $\eta$ that do not extend to maps $S\to \mathsf M_\Uplambda(\mathbb{G}_{\mathsf{trop}}^r)$ for {\it any} logarithmic structure on $S$ that restricts to the chosen one on $\eta$. We leave the details of this to the reader, as the it is orthogonal to the main results of the paper. 
		
		Another manifestation of the non-properness will come later in the paper, when we argue that after a subdivision, the space $\mathsf M_\Uplambda(\mathbb{G}_{\mathsf{trop}}^r)$ is representable by a non-proper Deligne--Mumford stack. 

	\end{warning}
	
	\begin{remark}
		Consider the moduli of logarithmic maps to $\mathbb{G}_{\mathsf{log}}^r$ from prestable curves of type $(g,n)$, without a specified contact profile, denoted $\mathfrak{M}_{g,n}(\mathbb{G}_{\mathsf{log}}^{r})$. We can view this as a sheaf of abelian groups on the \'etale site of $\mathfrak{M}_{g,n}$. We have a sheaf homomorphism
		\begin{equation*}
		\mathfrak{M}_{g,n}(\mathbb{G}^{r}_{\mathsf{log}})\to (\mathbb{Z}^{n}_{0})^r,
		\end{equation*}
		which sends a log map $f$ to the contact profile of $f$. 
	\end{remark}
	
	%We Choosing an isomorphism of $N$ with $\ZZ^r$, we can interpret the contact data $\underline{A}$ as a balanced set of integral vectors in $\mathbb{R}^r$, i.e. when viewed as a $r\times n$ matrix, the row sums are equal to $0$. This means for any projective toric variety $X$ and an identification of the cocharacter space $N_\RR$ with $\mathbb R^r$, we can interpret $\underline{A}$ as contact data for the moduli space of logarithmic maps from $n$-pointed curves to $X$. 
	
	The following proposition can essentially be found in~\cite[Section~2]{RSW17B}. We take the opportunity to present a detailed discussion.
	
	\begin{proposition}
		\label{prop:rigidmod}
		Let $X$ be a smooth projective toric variety with torus $\mathbb G_m^r$, equipped with its toric logarithmic structure. The subdivision $X \to \mathbb{G}_{\mathsf{log}}^r$ induces a subdivision $$\mathsf{M}_{\Uplambda}(X) \to \mathsf{M}_{\Uplambda}(\mathbb{G}_{\mathsf{log}}^r).$$ 
	\end{proposition}
	
	We prove this result via a related result at the Artin fan level. To state it, consider a family
	\begin{equation*}
	\begin{tikzcd}
	C \arrow[r, "f"] \arrow[d, swap,"\pi"] 
	& {\mathsf A_{X}} \\
	S,
	\end{tikzcd}
	\end{equation*}
	obtained via a moduli map $S\to \mathsf{M}_\Uplambda(\mathsf A_X)$.
	We stabilize the source curve to obtain a curve $\overline C\to S$, for example by using the moduli map $\fM_{g,n}\to \Mbar_{g,n}$. We observe that the only unstable components that {\it could} be contracted by $C\to \overline C$ are $2$-pointed $\mathbb P^1$-components. 
	
	By~\cite[Proposition~2.10]{ACGS15}, the map $C\to \mathsf A_X$ has a concrete description. Precisely, recall that $C$ admits a {\it generalized cone complex}, given by a colimit of cones as in~\cite[Appendix~B]{GS13}\footnote{See also~\cite{ACP} for details on the combinatorics and~\cite{ACMUW} for a discussion in the appropriate conceptual context.}. It is denoted $C_{\trop}$. The cited proposition tells us that the map $C\to \mathsf A_X$ is equivalent to the data of a piecewise linear map
	\[
	C_{\trop}\to \Sigma_X,
	\]
	where $\Sigma_X$ is the fan of $X$. Similarly, the curve $\overline C$ has an associated cone complex, and there is a map
	\[
	C_{\trop}\to\overline C_{\trop}.
	\]
	Now, the curves $C$ and $\overline C$ differ {\it only} by $2$-pointed $\mathbb P^1$-components. But by the stability in Definition~\ref{def: artin-fan-stability}, the composition
	\[
	C_{\trop}\to \Sigma_X\to \RR^r,
	\]
	factors through $\overline C_{\trop}$. Indeed, {\it linearity} is exactly the condition that the map factors through contracting the component, and {\it stability} is the condition that the map to $\Sigma_X$ {\it does not} factor through the contraction. 
	
	As previously noted, an $\mathbb R^r$-valued piecewise linear function on $\overline C_{\trop}$ is precisely $r$ sections of $H^0(\overline C,\overline M_{\overline C}^{\mathsf{gp}})$, we have produced a map:
	\[
	\mathsf{M}_{\Uplambda}(\mathsf A_X) \to \mathsf{M}_{\Uplambda}(\mathbb{G}_{\mathsf{trop}}^r).
	\]
	
	We now prove Proposition~\ref{prop:rigidmod}, via its ``tropical'' incarnation. 
	
	\begin{proposition}\label{prop:rigidmod-trop}
		Let $X$ be a smooth projective toric variety with torus $\mathbb G_m^r$, equipped with its toric logarithmic structure. The subdivision $\mathsf A_X\to  \mathbb{G}_{\mathsf{trop}}^r$ induces a subdivision
		$$
		\mathsf{M}_{\Uplambda}(\mathsf A_X) \to \mathsf{M}_{\Uplambda}(\mathbb{G}_{\mathsf{trop}}^r).
		$$ 
	\end{proposition}
	
	In the proof that follows, we use the tropical curve associated to a family of logarithmic curves over an atomic base. We will also use the description of sections of the characteristic abelian sheaf via piecewise linear functions. The reader may consult Remarks~\ref{rem: tropical-curves} and~\ref{rem: pl-functions}.
	
	\begin{proof}
		Consider a logarithmic scheme $S$ with a map $S\to \mathsf M_\Uplambda(\mathbb G_{\trop})$. Denote the pullback of $\mathsf{M}_{\Uplambda}(\mathsf A_X) \to \mathsf{M}_{\Uplambda}(\mathbb{G}_{\mathsf{trop}}^r)$ by $S'$. It suffices to prove the following:
		
		\noindent
		{\bf Claim.} The map $S'\to S$ is a subdivision. 
		
		\noindent
		{\sc Step I. Reduction to a toroidal base.} It suffices to treat the case where $S$ is atomic~\cite[Section~2]{AW}. The relevant part of atomicity is that $S$ has an Artin fan, and this Artin fan is in fact the one associated to a cone $\sigma$. We claim there is a factorization of $S\to \mathsf{M}_{\Uplambda}(\mathbb{G}_{\mathsf{trop}}^r)$ as
		\[
		S\to U_\sigma\to \mathsf{M}_{\Uplambda}(\mathbb{G}_{\mathsf{trop}}^r),
		\]
		where $S\to U_\sigma$ is strict and $U_\sigma\to \mathsf{M}_{\Uplambda}(\mathbb{G}_{\mathsf{trop}}^r)$ is logarithmically \'etale. If $\mathsf{M}_{\Uplambda}(\mathbb{G}_{\mathsf{trop}}^r)$ were representable one can show this easily using relative  Artin fans. In this non-representable context, we argue directly. 
		
		The map $S\to \mathsf{M}_{\Uplambda}(\mathbb{G}_{\mathsf{trop}}^r)$ consists of a logarithmic curve $C/S$ together with $r$ sections of $H^0(C,\overline M_C^{\mathsf{gp}})$. The latter datum can be packaged combinatorially. The curve $C/S$ determines a cone complex $C_{\trop}$ with a map $C_{\trop}\to \sigma$ and the $r$ sections can be understood as a $\mathbb R^r$-valued piecewise linear function on $C_{\trop}$. 
		
		By the results of~\cite{CCUW} we have a moduli map 
		\[
		\sigma\to \cM_{g,n}^{\trop}.
		\]
		We can pass to the associated Artin fans $\mathsf A_\sigma\to a^\star\cM_{g,n}^{\trop}$, and note that this is logarithmically \'etale. By using the strict and smooth map $\Mbar_{g,n}\to  a^\star\cM_{g,n}^{\trop}$ we define:
		\[
		U_\sigma:=\mathsf A_\sigma\times_{a^\star\cM_{g,n}^{\trop}} \Mbar_{g,n}.
		\]
		By the mapping property of this fiber product, we have a map
		\[
		S\to U_\sigma,
		\]
		and since $S\to \mathsf A_\sigma$ is strict by hypothesis so is this map. Furthermore, if $C_\sigma\to U_\sigma$ is the pullback of the universal curve, then the associated cone complex of this family, by construction, is $C_{\trop}\to \sigma$. Recall it is exactly this family that possess the map to $\mathbb R^r$. By reversing the combinatorial dictionary describing sections of the characteristic abelian sheaf, there is a diagram:
		\[
		\begin{tikzcd}
		C\arrow{d}\arrow{r}\arrow[rd, phantom, "\square"] & C_\sigma \arrow{r} \arrow{d} & \mathbb G_{\trop}^r\\
		S\arrow{r} & U_\sigma.&
		\end{tikzcd}
		\]
		Abstract nonsense reduces the claim to the case when $S\to \mathsf M_\Uplambda(\mathbb G_{\trop}^r)$ is logarithmically \'etale, i.e. $S = U_\sigma$. 
		
		We make a further reduction. The map $\mathsf M_\Uplambda(\mathbb G_{\trop}^r)\to\Spec(\CC)$ is logarithmically smooth. To see this, we consider the infinitesimal lifting criterion for logarithmically \'etale maps, applied to $\mathsf M_\Uplambda(\mathbb G_{\trop}^r)\to \Mbar_{g,n}$. Fix $T$ a logarithmic scheme and $T\to \mathsf M_\Uplambda(\mathbb G_{\trop}^r)$ is a morphism and let $T\subset T'$ be a strict square zero extension, together with a compatible map
		\[
		T'\to \Mbar_{g,n}.
		\]
		The lifting problem for $T'\dashrightarrow \mathsf M_\Uplambda(\mathbb G_{\trop}^r)$ depends only on the characteristic sheaf of the logarithmic structure on $T'$ and on the associated family of curves. Since the extension is strict, these data are unchanged, so there is a unique lift. Thus $\mathsf M_\Uplambda(\mathbb G_{\trop}^r)\to \Mbar_{g,n}$ is logarithmically \'etale.
		
		We therefore conclude that the composite map
		$$
		S\to \mathsf M_\Uplambda(\mathbb G_{\trop}^r)\to \Mbar_{g,n}\to \Spec \mathbb C
		$$ 
		is logarithmically smooth. In particular, it suffices to prove our claim above when the test scheme $S$ is logarithmically smooth over a point, i.e. ``toroidal''.
		
		\noindent
		{\sc Step II. The case of a toroidal base.} We take a logarithmically smooth and atomic scheme $S$ with $S\to  \mathsf M_\Uplambda(\mathbb G_{\trop}^r)$. One can now directly check that the map
		\[
		S'\to S
		\]
		is birational, logarithmically \'etale, and proper and is therefore a subdivision. We say a word about the verifications. For birationality, observe that because $S$ is logarithmically smooth over a point, the locus $S^\circ\subset S$ where the logarithmic structure is trivial is dense. The stacks $\mathsf M_\Uplambda(\mathsf A_X)$ and $\mathsf M_\Uplambda(\mathbb G_{\trop}^r)$ are identical on schemes with trivial logarithmic structure, so
		\[
		S'\to S
		\]
		carries $S^\circ$ isomorphically onto its image. To see the map is logarithmically \'etale, one can again apply the infinitesimal lifting criterion; the logic is very similar to the application in the previous step -- the lifting criterion concerns strict square-zero extensions, and the data to lift depend only on the characteristic sheaf. Finally, the map is proper because it satisfies the valuative criterion for properness. For this, one can directly apply the proof of~\cite[Proposition~4.7.2]{ACMW} with cosmetic changes; this proof treats the case where $\mathsf A_X\to \mathbb G^r_{\trop}$ is replaced by a subdivision of Artin fans. But the representability is not used in the verification of the valuative criterion, and our stability condition plays the role of the stability condition on the spaces denoted there as $\fM(\mathcal Y\to \mathcal X)$. 
	\end{proof}
	
	We prove the ``geometric'' version. The discussion in~\cite[Section~4]{AW} proves a very similar result.
	
	\subsubsection{Proof of Proposition~\ref{prop:rigidmod-trop}} 
	Consider the diagram
	\[
	\begin{tikzcd}
	\mathsf M_\Uplambda(X)\arrow{d}\arrow{r} & \mathsf M_\Uplambda(\mathbb G_{\mathsf{log}}^r)\arrow{d}\\
	\mathsf M_\Uplambda(\mathsf A_X)\arrow{r} & \mathsf M_\Uplambda(\mathbb G_{\mathsf{trop}}^r).
	\end{tikzcd}
	\]
	In Proposition~\ref{prop:rigidmod-trop} we proved the bottom horizontal was a subdivision, so it will suffice to show that the square is fibered. Let $G$ denote the fiber product $\mathsf M_\Uplambda(\mathsf A_X)\times_{\mathsf M_\Uplambda(\mathbb G_{\trop}^r)} \mathsf M_\Uplambda(\mathbb G_{\mathsf{log}}^r)$. The universal property of the fiber product immediately gives us a map $\mathsf M_\Uplambda(X)\to G$
	
	Conversely, the fiber product fitting into the square parameterizes unfilled diagrams:
	\[
	\begin{tikzcd}
	C\arrow{d}\arrow[dashed]{r} \arrow[bend left=25]{rr}&X \arrow{r}\arrow{d} &\mathsf A_X\arrow{d}\\
	\overline C\arrow{r}& \mathbb G_{\mathsf{log}}^r\arrow{r} &\mathbb G_{\trop}^r.
	\end{tikzcd}
	\]
	The right square is fibered so we get a filling of the dashed arrow {\it canonically}. We claim the map $C\to X$ is stable. In fact, this has been reverse engineered -- the curves $C$ and $\overline C$ differ only by semistable, i.e. $2$-pointed $\mathbb P^1$-components. By the stability condition, the two special points on such a component are mapped to different locally closed strata in $\mathsf A_X$, and therefore in $X$. The new semistable components therefore cannot be contracted in the map to $X$, and so are stable. We therefore have a moduli map $G\to\mathsf M_\Uplambda(X)$. 
	
	The maps in the two paragraphs above are clearly inverse, and the proposition follows. \qed

	\begin{remark}[Previous sightings]
		Maps to logarithmic tori have appeared in the literature in a few places already. In the genus $0$ case, the spaces were studied in \cite{RW19} and used to give a unified perspective on earlier work on the genus $0$ logarithmic Gromov--Witten theory of toric varieties~\cite{CMR14b,CS12,R15b}. It appeared in earlier work on stable maps to genus $1$~\cite{RSW17B} and is the heart of interactions between logarithmic geometry and the Picard variety~\cite{MW17,MW18}. 
	\end{remark}
	
	\subsection{Maps to rubber $\mathbb{G}_{\mathsf{log}}^{r}$ and $\mathbb{G}_{\mathsf{trop}}^{r}$, and the work of Marcus--Wise}
	We define ``rubber'' moduli spaces of maps with target $\mathbb{G}_{\mathsf{log}}^r$. In these spaces, maps to logarithmic tori are identified if they differ by the translation action of $\mathbb{G}_{\mathsf{log}}^r$. The rank $1$ case is the subject of~\cite{MW17}. 
	
	Perhaps interestingly, these spaces are easier to define than rubber spaces of relative stable maps~\cite{GV05,MP06}, since stability for ordinary and rubber maps amount to stability of the domain. 
	
	%In Remark~\ref{rem:stackwlog} we observe that our higher rank analogues are therefore also represented by algebraic stacks with logarithmic structure by reducing to the rank $1$ case. We also prove the analogous statement to Proposition ~\ref{prop:rigidmod} in the rubber case in Proposition~\ref{prop:rubbermod}. Again the key to the proof is a reduction to the $r=1$ case.
	
	\begin{definition}
		Let $\mathfrak{M}^{\mathsf{rub}}_{\Uplambda}(\mathbb{G}_{\mathsf{log}}^r)$ be the {\it moduli space of prestable logarithmic maps to rubber $\mathbb{G}_{\mathsf{log}}^r$} analogously to $\mathfrak{M}_{\Uplambda}(\mathbb{G}_{\mathsf{log}}^r)$, but with the following change: for a family of curves $\pi:C\to S$, replace the $r$ sections of $M_{C}^{\mathsf{gp}}$ by $r$ sections on $S$ of $\pi_\star(M_{C}^{\mathsf{gp}}) / M_S^{\mathsf{gp}}$. Let $\mathsf{R}_{\Uplambda}(\mathbb{G}_{\mathsf{log}}^r)$ be the {\it moduli space of stable logarithmic maps to rubber $\mathbb{G}_{\mathsf{log}}^{n}$} as the subcategory of $\mathfrak{M}^{\mathsf{rub}}_{\Uplambda}(\mathbb{G}_{\mathsf{log}}^r)$ where the curves are stable.
		
		The spaces $\mathfrak{M}^{\mathsf{rub}}_{\Uplambda}(\mathbb{G}_{\mathsf{trop}}^r)$  are defined similarly using $\pi_\star(\overline M_{C}^{\mathsf{gp}})$ and $\overline M_S^{\mathsf{gp}}$. The space $\mathsf{R}_{\Uplambda}(\mathbb{G}_{\mathsf{trop}}^r)$ is the locus where the curve is stable and the maps are balanced, see Definition~\ref{def: glog-def}.
	\end{definition}
	
	%\begin{remark}
	%Similar to the rigid case we define $\mathfrak{M}_{g,n}^{\mathsf{rub}}(\mathbb{G}_{\mathsf{log}}^r)$ to be the moduli space of prestable logarithmic maps to rubber $\mathbb{G}_{\mathsf{log}}^r$ where we do not specify the contact profile $\underline{A}$. Again we can view this as a sheaf of abelian groups on the \'etale site of $\mathfrak{M}_{g,n}$.
	%\end{remark}
	
	%In comparison with the non-rubber spaces, the striking fact about the rubber spaces is that they are representable by algebraic stacks equipped with logarithmic structure. 
	
	\subsubsection{Summarizing the work of Marcus--Wise}\label{rem:stackwlog}
	Unlike $\mathfrak{M}_{\Uplambda}(\mathbb{G}_{\mathsf{log}}^r)$ the space $\mathfrak{M}_{\Uplambda}^{\mathsf{rub}}(\mathbb{G}_{\mathsf{log}}^r)$ is representable by an algebraic stack with logarithmic structure. This can be understood using observations due to Marcus--Wise~\cite{MW17}. Let us collect the basic facts from this paper. 
	
	First, if $A$ is a contact order vector, the authors construct a mapping stack $\fM_{g,A}^{\mathsf{rub}}(\mathbb G_{\mathsf{trop}})$, consisting of prestable curves equipped with a section of the characteristic sheaf. The contact order is given by $A$. This is an Artin stack equipped with logarithmic structure denoted $\mathbf{Div}_{g,A}$. %We note that the rubber is crucial here; without it the stack is not algebraic. 
	
	Second, the authors show that there is an Abel--Jacobi map to the universal Picard stack
	\[
	{aj}_A:\fM_{g,A}^{\mathsf{rub}}(\mathbb G_{\mathsf{trop}})\to\mathbf{Pic},
	\]
	sending a section of the characteristic sheaf on a logarithmic curve $C$ to its associated $\mathcal O_C^\star$-torsor. The locus where the torsor is trivial was denoted there as $\mathbf{Div}_{g,A}(\mathcal O)$. It is exactly the space we call $\fM_{g,A}^{\mathsf{rub}}(\mathbb G_{\mathsf{log}})$. 
	
	Third, stable versions of these spaces may be defined by restriction along $\Mbar_{g,n}\to \fM_{g,n}$, i.e. imposing stability of the underlying curve, ignoring the logarithmic structure. The spaces are respectively denoted $\mathsf R_{g,A}(\mathbb G_{\mathsf{trop}})$ and $\mathsf R_{g,A}(\mathbb G_{\mathsf{log}})$. The stable versions are Deligne--Mumford stacks, rather than merely Artin stacks. 
	
	And fourth, they show that the natural maps $\fM_{g,A}^{\mathsf{rub}}(\mathbb G_{\mathsf{log}})\to\fM_{g,n}$ and $\mathsf R_\Uplambda(\mathbb G_{\mathsf{log}})\to\Mbar_{g,n}$ are logarithmic monomorphisms. For example, in the case of the stable map space, there is a factorization:
	\[
	\mathsf R_{g,A}(\mathbb G_{\mathsf{log}})\to \mathsf R_{g,A}(\mathbb G_{\mathsf{trop}})\to\Mbar_{g,n},
	\]
	where the first map is a strict closed immersion and the second map is a logarithmically \'etale monomorphism. In particular, $\mathsf R_{g,A}(\mathbb G_{\mathsf{trop}})\to\Mbar_{g,n}$ can be obtained composing a representable logarithmic modification, a generalized root construction, and an open immersion. An analogous statement holds for the prestable mapping spaces, though we will not need this. 
	
	The higher rank case behaves similarly on all fronts. We have the observation that:
	\begin{equation*}
	\mathfrak{M}_{\Uplambda}^{\mathsf{rub}}(\mathbb{G}_{\mathsf{log}}^r)=\mathfrak{M}_{g, A_1}^{\mathsf{rub}}(\mathbb{G}_{\mathsf{log}}) \times_{\mathfrak{M}_{g,n}} \dots \times_{\mathfrak{M}_{g,n}} \mathfrak{M}_{g, A_r}^{\mathsf{rub}}(\mathbb{G}_{\mathsf{log}}),
	\end{equation*}
	where the fiber product is taken in the category of fine and saturated logarithmic schemes. We also have the identification
	\begin{equation*}
	\textbf{Div}_{g, A_i}(\mathcal{O})\cong\mathfrak{M}_{g, A_i}^{\mathsf{rub}}(\mathbb{G}_{\mathsf{log}}),  \ \ \ \ i=1,\ldots, r.
	\end{equation*}
	Therefore $\mathfrak{M}_{\Uplambda}^{\mathsf{rub}}(\mathbb{G}_{\mathsf{log}}^r)$ is representable. Similarly $\mathsf{R}_\Uplambda(\mathbb G_{\mathsf{log}}^r)$ and $\mathsf{R}_\Uplambda(\mathbb G_{\mathsf{trop}}^r)$ are schematically representable, and carry similar relationships with the moduli of curves from the fourth point. 
	
	\subsection{Comparing rigid and rubber}
	The space $\mathfrak{M}_{\Uplambda}^{\mathsf{rub}}(\mathbb{G}_{\mathsf{log}}^r)$ should be thought of as the moduli space of logarithmic stable maps to $\mathbb{G}_{\mathsf{log}}^r$ upto shifts by $\mathbb{G}_{\mathsf{log}}^r$. This suggests the existence of a map:
	\begin{equation*}
	\upeta: \mathfrak{M}_{\Uplambda}(\mathbb{G}_{\mathsf{log}}^r) \to \mathfrak{M}_{\Uplambda}^{\mathsf{rub}}(\mathbb{G}_{\mathsf{log}}^r).
	\end{equation*}
	Precisely, a map from a curve $\pi:C\to S$ to $\mathbb G_{\mathsf{log}}^r$ is given by $r$ sections in $\pi_\star(M_C^{\mathsf{gp}})$. Sending each section $s$ to the corresponding section $\bar{s}$ of $\pi_\star(M_C^{\mathsf{gp}})/M_S^{\mathsf{gp}}$, we obtain the map above.

	Since $\mathbb{G}_{\mathsf{log}}^r$ is a group sheaf on the \'etale site of $\textbf{LogSch}$, one can make sense of $\mathbb{G}_{\mathsf{log}}^r$-torsors, and one expects that $\eta$ is such a torsor. We give the basic definitions of these objects. The discussion is essentially standard, but we provide some details, as the notion appears prominently in the text.
	
	\begin{definition}
		Let $P$ be a sheaf valued in sets on the strict \'etale site of a logarithmic scheme $S$. We assume we also have a fixed group action of a group sheaf $\mathscr{G}$ on $P$. We call $P$ a {\it $\mathscr{G}$-torsor} over $S$ if the following conditions are satisfied.
		\begin{enumerate}[(i)]
			\item if $P(U)$ is nonempty then the action of $\mathscr{G}(U)$ on $P(U)$ is simply transitive.
			\item There exists an open cover $\{U_i\}_{i\in I}$ of $X$ such that $P(U_i)$ is nonempty for all $i\in I$.
		\end{enumerate}
	\end{definition}
	
	\begin{remark}
		Condition (i) implies $P(U)$ can be $\mathscr{G}$-equivariantly identified with $\mathscr{G}(U)$. For an algebraic group $G$ if we set $\mathscr{G}=Mor(-, G)$ we recover the more familiar notion of a $G$-torsor.
	\end{remark}
	
	\begin{definition}\label{def: associated-torsor}
		Let $F$ be a functor valued in $\textbf{Sets}$ on $\textbf{LogSch}$. Let $\mathscr{G}$ be a group sheaf on the strict \'etale site of $\textbf{LogSch}$ with an action on $F$. Fix a logarithmic scheme $S$. We call a $\mathscr{G}$-invariant morphism $\pi: F\to S$ a $\mathscr{G}$-torsor if the following functor on the strict \'etale site of $S$, denoted $\mathcal F_\pi$, is a $\mathscr{G}$-torsor: an \'etale open $U\to S$ we assign the set $$\mathcal{F}_{\pi}(U):= \left\{x\in F(U) \mid \pi_U(x)= U\to S\right\}.$$
	\end{definition}
	
	The following proposition gives a transparent relationship between rigid and rubber, and illustrates the utility of working with the logarithmic torus.
	
	\begin{proposition}
		\label{prop:istorsor}
		The morphism $\eta$ is a $\mathbb{G}_{\mathsf{log}}^r$-torsor.
	\end{proposition}
	\begin{proof}
		We first assume $r=1$. For a logarithmic scheme $S$ with a map to $\textbf{Div}_{g, A}(\mathcal{O})$ we consider the following fiber product:
		\begin{equation*}
		\begin{tikzcd}
		S \times_{\textbf{Div}_{g, A}(\mathcal{O})} \mathfrak{M}_{\Uplambda}(\mathbb{G}_{\mathsf{log}}^r) \arrow[r, ] \arrow[rd, phantom, "\square"] \arrow[d,swap, "\delta"] 
		& \mathfrak{M}_{\Uplambda}(\mathbb{G}_{\mathsf{log}}^r) \arrow[d, "\eta"] \\
		S \arrow[r,]
		&  \textbf{Div}_{g, A}(\mathcal{O})
		\end{tikzcd}
		\end{equation*}
		We show that $\delta$ is a $\mathbb{G}_{\mathsf{log}}$-torsor. The map $S\to \textbf{Div}_{g, A}(\mathcal{O})$ determines a log curve $\pi\colon C\to S$ together with a section $s$ of $H^0(S,{\pi_\star(M_C^{\mathsf{gp}})}/M_S^{\mathsf{gp}})$. We consider the exact sequence of sheaves on $S$
		\begin{equation*}
		0\to M_{S}^{\mathsf{gp}}\to \pi_\star(M_C^{\mathsf{gp}}) \to {\pi_\star(M_C^{\mathsf{gp}})}/{M_{S}^{\mathsf{gp}}}\to 0.
		\end{equation*}
		From the map $\delta$ above, and in keeping with the notation of Definition~\ref{def: associated-torsor}, we have the functor that assigns to an \'etale open $U\to S$ 
		$$
		\mathcal{F}_{\delta}(U)=\left\{s'\in H^0(C_{\mid U}, M_{C_{\mid U}}^{\mathsf{gp}}) \mid \phi_U(s')=s\right\},
		$$ where $\phi$ is the second arrow in the preceding short exact sequence. We have a natural action of the group sheaf $\mathbb{G}_{\mathsf{log}}$ on $\mathcal{F}_{\delta}$. If $\mathcal{F}_{\delta}(U)$ is nonempty it is a coset in $H^0(C_{\mid U}, M_{C_{\mid U}}^{\mathsf{gp}})$. This implies the action of $H^0(U, M_U^{\mathsf{gp}})$ on $\mathcal{F}_{\delta}(U)$ is simply transitive. The section $s$ locally comes from a section of $\pi_\star(M_{C}^{\mathsf{gp}})$. Therefore we can find a cover by \'etale open sets $U$ such that $\mathcal{F}_{\delta}(U)$ is nonempty. The result follows.
	\end{proof}
	
	\begin{corollary}
		The morphism $\eta$ above restricts to a morphism
		\[
		\varepsilon\colon\mathsf M_\Uplambda(\mathbb G_{\mathsf{log}}^r)\to \mathsf{R}_\Uplambda(\mathbb G_{\mathsf{log}}^r)
		\]
		which is also a $\mathbb G_{\mathsf{log}}^r$-torsor. 
	\end{corollary}
	
	\begin{proof}
		Stability for maps to $\mathbb G_{\mathsf{log}}^r$ and for rubber maps to $\mathbb G_{\mathsf{log}}^r$ is the condition that the underlying curve is stable. So if we restrict the domain of $\eta$ from $\fM_\Lambda(\mathbb G_{\mathsf{log}}^r)$ to $\mathsf M_\Uplambda(\mathbb G_{\mathsf{log}}^r)$, the resulting map to $\fM^{\mathsf{rub}}_\Lambda(\mathbb G_{\mathsf{log}}^r)$ factors through $\mathsf{R}_\Uplambda(\mathbb G_{\mathsf{log}}^r)$. Since $\eta$ is a $\mathbb G_{\mathsf{log}}^r$-torsor, the corollary follows.  
	\end{proof}
	
	\begin{proposition}
		There is an exact sequence of of sheaves of abelian groups on the \'etale site of $\mathfrak{M}_{g,n}$:
		\begin{equation*}
		0\to \mathbb{G}_{\mathsf{log}}^r\to \mathfrak{M}_{g,n}(\mathbb{G}_{\mathsf{log}}^r)\to \mathfrak{M}_{g,n}^{\mathsf{rub}}(\mathbb{G}_{\mathsf{log}}^r)\to 0.
		\end{equation*}
	\end{proposition}
	
	\begin{remark}
		The exact sequence of~\cite[Theorem 4]{RW19} identifies
		\begin{equation*}
		\mathfrak{M}_{0,n}^{\mathsf{rub}}(\mathbb{G}_{\mathsf{log}})= \underline{\mathbb{Z}}_{0}^{n},
		\end{equation*}
		as sheaves on the \'etale site of $\mathfrak{M}_{0,n}$; the right hand side is the constant sheaf. This is a genus $0$ phenomenon, essentially coming from the fact that degree $0$ divisors and principal divisors on genus $0$ curves coincide. The $\ZZ_0$ above can be interpreted as the group of degree $0$ divisors supported on the marked points. In higher genus, this is certainly not true, and there is a virtual codimension $g$ condition cutting out the space of principal divisors inside the degree $0$ divisors.
	\end{remark}
	
	\section{Evaluation spaces}\label{sec: evaluation-spaces}

	We construct evaluation maps and evaluation spaces for moduli of maps to $\mathbb{G}_{\mathsf{log}}^r$. The torsor presentation of the rubber mapping space gives rise to a natural {\it rubber evaluation space}. We introduce rubber evaluation spaces, which play the key role in transferring integrals from rigid to rubber spaces. On the locus of non-degenerate maps the rubber evaluation has an elementary description, and the logarithmic torus allows us to extend this description to the boundary.
	
	\subsection{A review of evaluation maps} We review evaluation maps in logarithmic Gromov--Witten theory, focusing on certain subtleties that have not received an abundance of exposition. Much of the discussion applies to logarithmically smooth pairs, but we restrict to toric varieties. 
	
	Let $\mathsf M_\Uplambda(X)$ be the moduli of logarithmic stable maps to $X$ with numerical data $\Uplambda$. An $S$-point of the space is a diagram
	\[
	\begin{tikzcd}
	C\arrow{d}\arrow{r} & X\\
	S,
	\end{tikzcd}
	\]
	equipped with sections $p_i\colon S\to C$ corresponding to the marked points. In ordinary Gromov--Witten theory, composing $p_i$ with the universal map gives evaluations $\mathsf{ev}_i$ from $\mathsf M_\Uplambda(X)$ to $X$. 
	
	Logarithmic structures bring subtleties. Typically, the curve $C$ is given logarithmic structure by pulling back the natural logarithmic structure on the universal curve on the moduli space of (not necessarily minimal) logarithmic curves. Given a marked point, the section
	\[
	p_i\colon S\to C,
	\]
	is {\it not} logarithmic -- the logarithmic structure along the section is nontrivial relative to the base.%, e.g. it is nontrivial over the locus on $S$ with trivial logarithmic structure. 
	
	There are two ways out of the issue. One can {\it add} logarithmic structure on $S$ (i.e. to the $\mathsf M_\Uplambda(X)$) or {\it remove} logarithmic structure along the image of $p_i$. We opt for the latter. 
	
	\subsubsection{Removing logarithmic structure} The universal curve $\mathcal C$ over $\mathfrak M^{\mathsf{log}}_{g,n}$ has a divisor given by the image of the section $p_i$. The logarithmic structure on the universal curve $\mathcal C$ comes from a normal crossings divisor, and removing this divisor produces a new logarithmic structure on $\mathcal C$. The map $\mathcal C\to\mathfrak M^{\mathsf{log}}_{g,n}$ is still logarithmic, since the divisor is horizontal. We pull this back to any given family of curves, and refer to it as {\it removing the logarithmic structure along $p_i$}. Removing the logarithmic structure along all $p_i$ is known in the literature as the vertical part of the logarithmic structure.
	
	\subsubsection{Trivial contact order points} We construct evaluation maps. First, assume the contact order of $p_i$ is trivial, or geometrically, that the marking generically maps to the interior of $X$. Since the contact order is trivial, the morphism $C\to X$ remains logarithmic {\it after} removing logarithmic structure along $p_i$. We obtain morphisms:
	\[
	\mathsf{ev}_i\colon \mathsf M_\Uplambda(X)\to X. 
	\]
	
	\subsubsection{Nontrivial contact order points} When $p_i$ has nontrivial contact order, we use the following trick. The contact order at $p_i$ determines a vector $v_i$ in the cocharacter lattice $N$. The vector $v_i$ lies in the relative interior of a cone $\sigma_i$ in the fan of $X$. Let $W_i$ be the closed toric stratum dual to $\sigma_i$. 
	
	Now, there is a rational map $X\dashrightarrow W_i$, since the torus in $W_i$ is naturally a quotient of the torus in $X$. We can blow up $X$ at strata that do not intersect $W_i$ such that the rational map extends to a morphism $X'\rightarrow W_i$. Note that this map is torus equivariant, so automatically logarithmic. We have a morphism on moduli spaces:
	%indeed the normal bundle to $W_i$ is naturally identified with an invariant open subset of $W_i$. By passing to a toric blowup at such ``far away'' strata, we can assume the rational map extends to a morphism. Since the construction below will only involve maps in a neighborhood of $W_i$, the assumption can be made without any loss of generality.By passing to mapping spaces, we now have a morphism
	\[
	\mathsf M_\Uplambda(X')\to \mathsf M_\Uplambda(W_i).
	\]
	The induced discrete data on the right is obvious for everything except the contact order. For that, note there is a projection from the cocharacter lattice of $X$ to that of $W_i$. The image of the contact order for each point determines the contact order for this point on the right hand side. 
	
	\begin{lemma}
		Let $[C\to X']$ be a logarithmic stable map with discrete data $\Uplambda$. In the composite map $[C\to X'\to W_i]$ the point $p_i$ has contact order $0$. 
	\end{lemma}
	
	\begin{proof}
		The cocharacter space of $W_i$ is the quotient of the cocharacter space of $X$ by the linear span of the cone $\sigma_i$. The contact order for $p_i$ lies in $\sigma_i$, so it is $0$ in the quotient. 
	\end{proof}
	
	By the lemma, the previous discussion about evaluations with trivial contact order, and the morphism $\mathsf M_\Uplambda(X')\to \mathsf M_\Uplambda(W_i)$ we have a logarithmic evaluation map
	\[
	\mathsf{ev}_i'\colon \mathsf M_\Uplambda(X')\to W_i.
	\]
	This will actually suffice for our purposes, as we are always prepared to blowup up the target further. The following is recorded for future reference:
	\begin{proposition}
		The logarithmic evaluation map $\mathsf{ev}_i'$ descends to
		\[
		\mathsf{ev}_i\colon \mathsf M_\Uplambda(X)\to W_i.
		\]
	\end{proposition}
	
	\begin{proof}
		Consider the composition $\mathsf M_\Uplambda(X')\to X'\to X$ of $\mathsf{ev}_i'$, viewed as a map to $X'$, with the blowup. We argued that the image of $\mathsf{ev}_i'$ lies in $W_i$. The composed map descends to the evaluation from $\mathsf M_\Uplambda(X)$ to $X$. However, the map $\mathsf M_\Uplambda(X')\to \mathsf M_\Uplambda(X)$ is proper and surjective, so the scheme theoretic image of $\mathsf M_\Uplambda(X)$ in $X$ is contained inside the image of $W_i$ inside $X$, as required.  
	\end{proof}
	
	\begin{remark}[Dimensionally stable evaluation spaces]
		By blowing up the target $X$, one can assume that the strata $W_i$ associated to the markings are either divisors or $X$. It is convenient to do this, as the evaluation space does not change dimension upon further blowup. There is no loss of generality, since evaluation constraints can always be pulled back to such a blowup. 
	\end{remark}

	%Recall that the contact data $\underline{A}$ induces a balanced set of integral vectors $v_1, ..., v_m$ in $\mathbb{R}^r$. For the remainder of this section we fix a smooth projective $(\mathbb{C}^{*})^{n}$-toric variety $X$ whose fan contains the rays generated by the $v_i$ for $i=1, ..., m$. Our goal is to lift these evaluation spaces and evaluation maps to the moduli spaces of maps to $X$. For the moduli of maps to rigid $X$ these evaluation objects are used through out the literature on logarithmic GW-theory (See??????). The novelty here is that we construct an analogous evaluation space and evaluation map for the moduli space of maps to rubber $X$.
	
	\subsection{Evaluating to logarithmic tori} We construct evaluations for logarithmic maps to $\mathbb G_{\mathsf{log}}^r$. Each marking $p_i$ with nontrivial contact order determines a non-constant $1$-parameter subgroup by taking the span of the contact vector and tensoring the resulting map of lattices with $\mathbb G_{\mathsf{log}}$:
	\begin{equation*}
	\mathbb{G}_{\mathsf{log}} \hookrightarrow \mathbb{G}_{\mathsf{log}}^r.
	\end{equation*}
	We can therefore define:
	
	\begin{definition}[The evaluation space]
		The {\it rigid evaluation space} of $\mathsf{M}_{\Uplambda}(\mathbb{G}_{\mathsf{log}}^r)$ for $p_i$ is the quotient $\mathbb{G}_{\mathsf{log}}^r/\mathbb{G}_{\mathsf{log}}$ by the subtorus above. If $p_i$ has trivial contact the rigid evaluation space is $\mathbb{G}_{\mathsf{log}}^r$ itself. The {\it consolidated rigid evaluation space} $\mathsf{Ev}_{\Uplambda}(\mathbb{G}_{\mathsf{log}}^r)$ is the product of these spaces over all markings.
	\end{definition}
	
	We can define evaluation maps similarly. Let us first assume that we have marked point $p_i$ where the contact order is equal to $0$. Given a family of logarithmic curves $C$ over $S$, with a map to $\mathbb{G}_{\mathsf{log}}^r$, we may remove the logarithmic structure along the section corresponding to $p_i$. By composing the natural section $S\to C$ with the map to $\mathbb{G}_{\mathsf{log}}^r$ we obtain an $S$-valued point of $\mathbb{G}_{\mathsf{log}}^r$. Globally, we have:
	\[
	\mathsf{ev}_i:\mathsf M_\Lambda(\mathbb G_{\mathsf{log}}^r)\to \mathbb G_{\mathsf{log}}^r.
	\]
	
	In order to record some consistency between the evaluations for toric targets and for the logarithmic torus, and to provide some geometric intuition, let us give another view on these maps. We can use the fact that evaluation maps already exist for $\mathsf{M}_{\Uplambda}(X)$, where $X$ is a subdivision of $\mathbb G^r_{\mathsf{log}}$ and $\Lambda$ is the induced data. 
	
	Let us again first assume the contact order is $0$. Consider the following diagram
	\begin{equation*}
	\begin{tikzcd}
	\mathsf{M}_{\Uplambda}(\mathbb{G}_{\mathsf{log}}^r) \arrow[r, dashed]
	& \mathbb{G}_{\mathsf{log}}^r  \\
	\mathsf{M}_{\Uplambda}(X) \arrow[u] \arrow[r]
	&  X \arrow[u]
	\end{tikzcd}
	\end{equation*}
	The contact order is $0$, so we remove the logarithmic structure along the marked point $p_i$. The dashed arrow is filled in, as above, by composing the natural section from $\mathsf{M}_{\Uplambda}(\mathbb{G}_{\mathsf{log}}^r)$ to the universal curve, with the morphism to $\mathbb G_{\mathsf{log}}^r$. The resulting square commutes and gives the evaluation morphism:
	\begin{equation*}
	\mathsf{ev}_i\colon\mathsf{M}_{\Uplambda}(\mathbb{G}_{\mathsf{log}}^r)\to \mathbb{G}_{\mathsf{log}}^r.
	\end{equation*}
	Intuitively, it is obtained simply by ``blowing down'' both sides of evaluation map in the schematically representable case. 
	
	Now if the $i^{\mathrm{th}}$ contact order is nonzero we can use the projection trick from the previous section. Precisely, we have seen that the contact order of $p_i$ determines a $1$-parameter subgroup $\mathbb G_{\mathsf{log}}\hookrightarrow\mathbb G^r_{\mathsf{log}}$. By projecting a curve onto the quotient by this subgroup, the contact order at $p_i$ becomes $0$. By applying the previous construction we get an evaluation morphism
	\begin{equation*}
	\mathsf{ev}_i\colon\mathsf{M}_{\Uplambda}(\mathbb{G}_{\mathsf{log}}^r)\to \mathbb{G}_{\mathsf{log}}^r / \mathbb{G}_{\mathsf{log}}.
	\end{equation*}
	We will combine these to get a consolidated evaluation map:
	\[
	\mathsf{ev}\colon \mathsf{M}_{\Uplambda}(\mathbb{G}_{\mathsf{log}}^r)\to \mathsf{Ev}_\Uplambda(\mathbb G^r_{\mathsf{log}}).
	\]
	
	%\begin{remark}
	%These evaluation morphisms are independent of the choice of toric variety $X$
	%\end{remark}
	%\begin{remark}
	%contact order $0$ evaluation splits injection in exact sequence
	%\end{remark}
	%\begin{remark}
	%tropical picture
	%\end{remark}
	
	\subsection{Effectivity of the torus action} We introduce an assumption for the remainder of the paper; its helps to condense some diagrams that appear later. Consider the action of $\mathbb G_{\mathsf{log}}^r$ on $\mathsf{Ev}_\Uplambda(\mathbb G^r_{\mathsf{log}})$. It is induced by a group action of the lattice $N$ on the lattice $N'$ associated to $\mathsf{Ev}_\Uplambda(\mathbb G^r_{\mathsf{log}})$. Explicitly:
	\begin{equation*}
	N'= N / \mathbb{Z} w_1 \oplus \cdots \oplus N / \mathbb{Z} w_n
	\end{equation*}
	where $w_i$ is the primitive integral vector in the direction of the contact order $v_i$ for $i=1,\ldots, n$. And the action of $N$ on $N'$ is induced by the quotient group homomorphism $\phi: N \to N'$. The orbit of the origin (i.e the image of $\phi$) is a sublattice $L$ of $N'$. 
	
	\noindent
	{\bf Effectivity assumption.} {\it The lattice $L$ above is a torsor for $N$, i.e. the homomorphism $\phi$ has trivial kernel. Furthermore by replacing $N'$ with a finite index sublattice, assume that $N' / \phi(N)$ is free. }
	
	Passing to a finite index sublattice amounts to a root construction, and does not affect the Gromov--Witten theory~\cite{AW}. A typical case where the effectivity {\it fails} is when the target is $\mathbb P^1$, and all marked points have non-zero contact order. The evaluation space is a point, so effectivity fails. 
	
	\subsubsection{Justifying the effectivity}\label{rem:eff}
	Effectivity implies we have a split short exact sequence:
	\begin{equation*}
	0\to N\to N'\to N' / \phi(N) \to 0.
	\end{equation*}
	If we tensor this with $\mathbb{G}_m$ we get that the dense torus of $X$ injects into the dense torus of $\mathsf{Ev}_\Uplambda(X)$. Therefore, the action of the torus of $X$ on the torus of $\mathsf{Ev}_{\Uplambda}(X)$ is effective. We only ever care about the effective case, as if we had a toric variety where this action was not effective, its primary logarithmic Gromov--Witten cycles would vanish, as we now argue. Fix a cohomology class $\gamma$ on some subdivision $\mathsf{Ev}^\dagger_\Uplambda(X)$. The torus $N\otimes\mathbb G_m$ maps to $\mathsf{Ev}^\dagger_\Uplambda(X)$ via the homomorphism above. In the absence of effectivity, this map has positive dimensional fibers, and contains a copy of $\mathbb G_m$ that is contracted in $\mathsf{Ev}^\dagger_\Uplambda(X)$. Choose such a $\mathbb G_m$. Up to birational modifications, we can write $X$ as $Y\times\mathbb P^1$, where $\mathbb P^1$ is the compactification of this $\mathbb G_m$. The evaluation factors as
	\[
	\mathsf M_\Uplambda(Y\times\mathbb P^1)\to \mathsf M_\Uplambda(Y)\to \mathsf{Ev}^\dagger_\Uplambda(Y).
	\]
	We identify $\gamma$ with a cohomology class on a subdivision of $\mathsf{Ev}_\Uplambda(Y)$. The vanishing follows from the product formula~\cite{Herr,R19b}, together with the vanishing of the pushforward of $[\mathsf M_\Uplambda(\mathbb P^1)]^{\mathsf{vir}}$ to $\Mbar_{g,n}$.
	
	\subsection{The rubber evaluation space} The group $\mathbb G^r_{\mathsf{log}}$ acts on itself and on any group quotient of it. Therefore $\mathbb G^r_{\mathsf{log}}$ acts on the consolidated (rigid) evaluation space by the diagonal action on the factors of its presentation as a product.  
	
	\begin{definition}
		The {\it consolidated rubber evaluation space $\mathsf{Ev}^{\mathsf{rub}}_{\Uplambda}(\mathbb{G}_{\mathsf{log}}^r)$} is:
		\begin{equation*}
		\mathsf{Ev}^{\mathsf{rub}}_{\Uplambda}(\mathbb{G}_{\mathsf{log}}^r):= \mathsf{Ev}_{\Uplambda}(\mathbb{G}_{\mathsf{log}}^r)/\mathbb{G}_{\mathsf{log}}^r,
		\end{equation*}
		where the quotient is taken in the category of \textbf{AbGrp}-valued sheaves on the \'etale site of $\textbf{LogSch}$.
	\end{definition}
	
	By Proposition~\ref{prop:istorsor}, the rigid-to-rubber map on the mapping spaces is a $\mathbb{G}_{\mathsf{log}}^r$-torsor. The consolidated rigid evaluation map is equivariant for this group, so it descends to a rubber evaluation:
	\begin{equation*}
	\mathsf{ev}^{\mathsf{rub}}\colon\mathsf{R}_\Uplambda(\mathbb{G}_{\mathsf{log}}^r) \to \mathsf{Ev}^{\mathsf{rub}}_{\Uplambda}(\mathbb{G}_{\mathsf{log}}^r).
	\end{equation*}
	
	We now define an analogous rubber evaluation space $\mathsf{Ev}^{\mathsf{rub}}_{\Uplambda}(X)$ for $\mathsf{R}_{\Uplambda}(X)$ for representable targets $X\to \mathbb G_{\mathsf{log}}^r$. Fix such a toric variety $X$.
	\begin{definition}
		\label{def:Quotfan}
		We define a rubber evaluation space to be a representable smooth subdivision $Q$ of $\mathsf{Ev}^{\mathsf{rub}}_{\Uplambda}(\mathbb{G}_{\mathsf{log}}^r)$ such that there exists a smooth subdivision $B$ of $\mathsf{Ev}_{\Uplambda}(X)$ with a combinatorially flat morphism $b: B \to Q$ such that the following diagram commutes:
		\begin{equation*}
		\begin{tikzcd}
		B \arrow[r, ] \arrow[d,swap,"b"] 
		& \mathsf{Ev}_{\Uplambda}(\mathbb{G}_{\mathsf{log}}^r) \arrow[d,] \\
		Q \arrow[r,]
		&  \mathsf{Ev}_{\Uplambda}^{\mathsf{rub}}(\mathbb{G}_{\mathsf{log}}^r).
		\end{tikzcd}
		\end{equation*}
	\end{definition}
	\begin{remark}
		Rubber evaluation spaces always exist due to the following construction. Start with any representable subdivision $Q$ of $\mathsf{Ev}_{\Uplambda}^{\mathsf{rub}}(\mathbb{G}_{\mathsf{log}}^r)$. Then we consider $B:= Q \times_{\mathsf{Ev}_{\Uplambda}^{\mathsf{rub}}(\mathbb{G}_{\mathsf{log}}^r)} \mathsf{Ev}_{\Uplambda}(X)$ with its morphism $b\colon B\to Q$. Then by applying semistable reduction we can assume $b$ is combinatorially flat with smooth source and target.
	\end{remark}
	One should think of a rubber evaluation space as a quotient of the rigid evaluation space by a subtorus of its dense torus. The subtorus is given by the inclusion
	\[
	\mathbb{G}_m \otimes N \to \mathbb{G}_m \otimes N'
	\]
	The exact choice of rubber evaluation space is not relevant to our paper but an explicit construction can be obtained from the construction of~\cite{KSZ91}. Further combinatorial details about this construction can be found in work of Molcho~\cite{Mol16}.
	\begin{proposition}\label{prop:Glogcart}
		The following diagram is cartesian:
		\begin{equation*}
		\begin{tikzcd}
		\mathsf{M}_{\Uplambda}(\mathbb{G}_{\mathsf{log}}^r) \arrow[r, ] \arrow[d,"\varepsilon",swap] \arrow[rd, phantom, "\square"]
		& \mathsf{Ev}_{\Uplambda}(\mathbb{G}_{\mathsf{log}}^r) \arrow[d,] \\
		\mathsf{R}_\Uplambda(\mathbb{G}_{\mathsf{log}}^r) \arrow[r,]
		&  \mathsf{Ev}_{\Uplambda}^{\mathsf{rub}}(\mathbb{G}_{\mathsf{log}}^r).
		\end{tikzcd}
		\end{equation*}
	\end{proposition}
	\begin{proof}
		We show that for any base change of $\varepsilon$ via a map $S\to \mathsf{Ev}_\Uplambda^{\mathsf{rub}}(\mathbb G_{\mathsf{log}}^r)$ from a logarithmic scheme, the resulting outer square is cartesian:
		\begin{equation}
		\label{eq:fiberS}
		\begin{tikzcd}
		S \times_{\mathsf{R}_\Uplambda(\mathbb{G}_{\mathsf{log}}^r)}\mathsf{M}_{\Uplambda}(\mathbb{G}_{\mathsf{log}}^r) \arrow[r, ] \arrow[d,swap,"\gamma"]
		& \mathsf{Ev}_{\Uplambda}(\mathbb{G}_{\mathsf{log}}^r) \arrow[d,] \\
		S \arrow[r,]
		&  \mathsf{Ev}_{\Uplambda}^{\mathsf{rub}}(\mathbb{G}_{\mathsf{log}}^r).
		\end{tikzcd}
		\end{equation}
		Since $\gamma$ is $\mathbb{G}_{\mathsf{log}}^{n}$-torsor we have a cover of S by \'etale open sets $U \to S$ such that
		\begin{equation*}
		U \times_{\mathsf{R}_\Uplambda(\mathbb{G}_{\mathsf{log}}^r)}\mathsf{M}_{\Uplambda}(\mathbb{G}_{\mathsf{log}}^r)= U \times \mathbb{G}_{\mathsf{log}}^r,
		\end{equation*}
		and the restriction of $\gamma$ to $U \times \mathbb{G}_{\mathsf{log}}^r$  is the projection on to the first factor. The resulting diagram is 
		\begin{equation*}
		\begin{tikzcd}
		U\times \mathbb{G}_{\mathsf{log}}^r \arrow[r, ] \arrow[d,swap,"\gamma"]
		& \mathsf{Ev}_{\Uplambda}(\mathbb{G}_{\mathsf{log}}^r) \arrow[d,] \\
		U \arrow[r,]
		&  \mathsf{Ev}_{\Uplambda}^{\mathsf{rub}}(\mathbb{G}_{\mathsf{log}}^r),
		\end{tikzcd}
		\end{equation*}
		which is clearly cartesian. By gluing we get that the diagram in equation \eqref{eq:fiberS} is cartesian.
	\end{proof}
	We now lift this statement from $\mathbb{G}_{\mathsf{log}}^r$ to the toric variety $X$. 
	\begin{proposition}
		\label{prop:toriccart}
		There exists a cartesian square of fine and saturated logarithmic schemes:
		\begin{equation*}
		\begin{tikzcd}[row sep=scriptsize, column sep=scriptsize]
		X_1 \arrow[rr, crossing over] \arrow[dd] \arrow[rrdd, phantom, "\square"] & & X_3 \arrow[dd] \\
		& & &  \\
		X_2 \arrow[rr] & & X_4 \\
		\end{tikzcd}
		\end{equation*}
		where $X_1, X_2, X_3, X_4$ are representable subdivisions of the $4$ spaces appearing in Proposition~\ref{prop:Glogcart} such that the following diagram commutes:
		\begin{equation*}
		\label{eq:cubediag}
		\begin{tikzcd}[row sep=scriptsize, column sep=scriptsize]
		& X_1 \arrow[dl] \arrow[rr] \arrow[dd] & & X_3 \arrow[dl] \arrow[dd] \\
		\mathsf{M}_{\Uplambda}(\mathbb{G}_{\mathsf{log}}^r) \arrow[rr, crossing over] \arrow[dd] & & \mathsf{Ev}_{\Uplambda}(\mathbb{G}_{\mathsf{log}}^r) \\
		& X_2 \arrow[dl] \arrow[rr] & & X_4 \arrow[dl] \\
		\mathsf{R}_\Uplambda(\mathbb{G}_{\mathsf{log}}^r)  \arrow[rr] & & \mathsf{Ev}_{\Uplambda}^{\mathsf{rub}}(\mathbb{G}_{\mathsf{log}}^r) \arrow[from=uu, crossing over]\\
		\end{tikzcd}
		\end{equation*}
		Furthermore we require $X_3$ to be a subdivision of $\mathsf{Ev}_{\Uplambda}(X)$ with the morphism $X_3\to X_4$ being flat. The properties can all be ensured while also demanding $X_3$ and $X_4$ are smooth. 
	\end{proposition}
	
	\begin{proof}
		We define $X_4$ to be a rubber evaluation space as in Definition~\ref{def:Quotfan}. Therefore we have $X_3$ a subdivision of $\mathsf{Ev}_{\Uplambda}(X)$ with a flat morphism to $X_4$. Both $X_3$ and $X_4$ are smooth. Now define $X_2$ as:
		\begin{equation*}
		X_2 := \mathsf{R}_\Uplambda(\mathbb{G}_{\mathsf{log}}^r) \times_{\mathsf{Ev}_{\Uplambda}^{\mathsf{rub}}(\mathbb{G}_{\mathsf{log}}^r)} X_4. 
		\end{equation*}
		Similarly for $X_1$:
		\begin{equation*}
		X_1:= \mathsf{M}_{\Uplambda}(\mathbb{G}_{\mathsf{log}}^r) \times_{\mathsf{Ev}_{\Uplambda}(\mathbb{G}_{\mathsf{log}}^r)} X_3.
		\end{equation*}
		We get a map from $X_1\to X_2$ which gives a commutative cube diagram as in \eqref{eq:cubediag}. Since both the square in Proposition~\ref{prop:Glogcart} and the square involving $\mathsf{M}_{\Uplambda}(\mathbb{G}_{\mathsf{log}}^r)$, $\mathsf{Ev}_{\Uplambda}(\mathbb{G}_{\mathsf{log}}^r)$, $X_1$ and $X_3$ are cartesian, their composition is cartesian. This implies the outer rectangle in
		\begin{equation*}
		\begin{tikzcd}[row sep=scriptsize, column sep=scriptsize]
		X_1 \arrow[rr, crossing over] \arrow[dd] & & X_2 \arrow[rr] \arrow[dd] \arrow[rrdd, phantom, "\square"] & & \mathsf{R}_\Uplambda(\mathbb{G}_{\mathsf{log}}^r) \arrow[dd]\\
		& & &  \\
		X_3 \arrow[rr] & & X_4 \arrow[rr] & & \mathsf{Ev}_{\Uplambda}^{\mathsf{rub}}(\mathbb{G}_{\mathsf{log}}^r) \\
		\end{tikzcd}
		\end{equation*}
		is cartesian, so the left square is too. The representability of $X_2$ comes from being a subdivision of a space that is representable by a stack (c.f. Section~\ref{rem:stackwlog}). As a fiber product of representable spaces we get that $X_1$ is representable. 
	\end{proof}
	
	\section{Logarithmic GW {\it \&} DR cycles}\label{sec: main-theorems}
	
	We continue to let $X$ be a toric variety and let $\mathsf A_X$ be its Artin fan. Logarithmic Gromov--Witten theory produces logarithmic mapping stacks to $X$ and to $\mathsf A_X$ respectively~\cite{AC11,Che10,GS13}. The stack of maps to $X$ is always strict over the stack of maps to $\mathsf A_X$. We consider {\it subdivisions} of these mapping spaces. But we will only ever consider subdivisions that arise via strict maps:
	\[
	\mathsf M^\dagger_\Uplambda(X)\to \mathsf M^\dagger_\Uplambda(\mathsf A_X),
	\]
	where $\mathsf M^\dagger_\Uplambda(\mathsf A_X)$ is a subdivision of the usual logarithmic mapping space, with the stability condition of Section~\ref{sec: maps-to-tvs}. The space $\mathsf M^\dagger_\Uplambda(X)$ is the pullback of $\mathsf M_\Uplambda(X)$ along this subdivision. Each $\mathsf M^\dagger_\Uplambda(X)$ has a virtual class in Chow homology, compatible under pushforward along subdivisions~\cite[Section~3]{R19}. 
	
	We consider the analogous class of subdivisions of rubber mapping spaces. Recall the space of {rubber} stable maps to the logarithmic torus $\mathbb G_{\mathsf{log}}^r$ from Section~\ref{sec: maps-to-tvs}. As in that section, this space has a map to the space of rubber maps to $\mathbb G_{\mathsf{trop}}^r$. The latter is equipped with a natural stability condition. As above we will consider subdivisions of both spaces, but will always demand a strict map:
	\[
	\mathsf{R}^\dagger_{\Uplambda}(\mathbb G_{\mathsf{log}}^r)\to \mathsf{R}^\dagger_\Uplambda(\mathbb G_{\mathsf{trop}}^r).
	\]
	The subdivisions $\mathsf{R}^\dagger_{\Uplambda}(\mathbb G_{\mathsf{log}}^r)$ carry virtual classes, given by the Fulton--Macpherson intersection class from its presentation as an intersection of regular embeddings in a product of Picard varieties over an open in a subdivision of $\Mbar_{g,n}$, see~\cite[Section~4]{MR21}. 
	
	And finally, we will also consider the analogous class of subdivisions for maps to $\mathbb G^r_{\sf log}$ and $\mathbb G_{\trop}^r$. 
	
	\subsection{Conventions on notation and the minimal models} In order to avoid an overabundance of decorations, we adopt the following notation. We use subdivision-insensitive notation from here forward, and explicitly note when spaces have to be replaced by subdivisions, rather than decorate the spaces themselves; decorations are used only when two subdivisions need to be compared. 
	
	So from here forward, the notation $\mathsf M_\Uplambda(X)$ stands for a subdivision of the usual moduli space of logarithmic stable maps. Note that subdivisions of the rubber variant are denoted $\mathsf{R}_\Uplambda(\mathbb G_{\mathsf{log}}^r)$ rather than $\mathsf{R}_\Uplambda(X)$. There is no meaningful space of rubber maps to a general $X$, so it is not really appropriate to use the symbol ``$X$'' here. The symbol $\mathsf M$ will be replaced by $\mathfrak M$ for prestable variants, where the stability condition is dropped. The moduli space of prestable logarithmic curves is denoted $\mathfrak M_{g,n}^{\mathsf{log}}$. 
	
	Given a toric variety $X$, among the subdivisions of the mapping space we consider, there is a trivial subdivision, namely, the usual space of maps to $X$ from~\cite{AC11,Che10,GS13}. We denote this $\mathsf{M}_\Uplambda^{\mathsf{min}}(X)$. The superscript stands for \textit{minimal model}. Similarly, there is a distinguished functor $\mathsf{M}_\Uplambda^{\mathsf{min}}(\mathbb G_{\mathsf{log}}^r)$ on logarithmic schemes. It is typically schematically non-representable. We use the superscript similarly for maps to Artin fans $\mathsf A_X$, to the tropical torus $\mathbb G_{\mathsf{trop}}^r$, and for the spaces of rubber maps to $\mathbb G_{\mathsf{log}}^r$ and $\mathbb G_{\mathsf{trop}}^r$. 
	
	%Finally, we remind the reader that spaces of maps to Artin fans or to tropical tori are considered with their ``stability condition'', described in Section~\ref{sec: maps-to-tvs}.
	
	\subsection{Logarithmic Gromov--Witten cycles}\label{sec: log-cycle} We have a forgetful morphism $\varphi\colon \mathsf M_\Uplambda(X)\to \Mbar_{g,n}$. Given a class $\gamma$ in the Chow ring of $\mathsf{Ev}_\Uplambda(X)$, we produce a class in the ordinary Chow ring $\Mbar_{g,n}$ by the pull/push formula:
	\[
	\varphi_\star\left(\mathsf{ev}^\star(\gamma)\cap [\mathsf M_\Uplambda(X)]^{\mathsf{vir}} \right).
	\]
	We explain how to lift this to logarithmic Chow. Recall the ring is the colimit
	\[
	\mathsf{logCH}^\star(\Mbar_{g,n}) \colonequals \varinjlim_{\Mbar_{g,n}^\dagger\to\Mbar_{g,n}} \mathsf{CH}^\star(\Mbar_{g,n}^\dagger),
	\]
	taken over smooth subdivisions, with transitions given by pullback. Given $\Mbar_{g,n}^\dagger\to\Mbar_{g,n}$ a subdivision with smooth domain, we can perform a fine and saturated logarithmic base change along $\varphi$ to produce a map from a subdivision of $\mathsf M_\Uplambda(X)$ to $\Mbar_{g,n}^\dagger$. The virtual class produces an element in the Chow ring of this blowup. 
	
	In Chow cohomology, pullback along a blowup sequence with smooth centers is split injective, so we have an inclusion
	\[
	\mathsf{logCH}^\star(\Mbar_{g,n}) := \varinjlim \mathsf{CH}^\star(\Mbar_{g,n}^\dagger)\subset \varprojlim \mathsf{CH}^\star(\Mbar_{g,n}^\dagger):=\mathsf{logCH}_\star(\Mbar_{g,n}).
	\]
	The colimit is recognized inside the limit as those systems of compatible classes that are eventually related by pullback. We show the logarithmic Gromov--Witten class above is such a class.
	
	Any subdivision of $\mathsf M_\Uplambda(X)$ that maps to the fine and saturated pullback
	\[
	\Mbar_{g,n}^\dagger\times_{\Mbar_{g,n}}\mathsf M_\Uplambda(X)
	\] 
	also maps to $\Mbar_{g,n}^\dagger$. Note that this pullback is strict over the analogous fiber product with $X$ replaced by $\mathsf A_X$, so it falls under the class of subdivisions we are considering. As noted earlier, these subdivisions carry virtual classes~\cite[Section~3]{R19}. It is proved in loc. cit. that the classes are related by proper pushforward. The evaluation class $\mathsf{ev}^\star(\gamma)$ can be pulled back along further subdivisions. If we cap with the virtual class of each one of these subdivisions and push forward, we get a class on every smooth subdivision of $\Mbar_{g,n}$. This defines an element of the inverse limit above by pushforward compatibility. It is called the {\it the logarithmic Gromov--Witten cycle.}
	
	\noindent
	{\bf Assertion.} {\it The logarithmic Gromov--Witten cycle in $\varprojlim \mathsf{CH}^\star(\Mbar_{g,n}^\dagger)$ from the construction above is contained in the colimit of Chow rings, i.e. in $\mathsf{logCH}^\star(\Mbar_{g,n})$.}
	
	In their forthcoming survey, Herr--Molcho--Pandharipande--Wise explain why the system of classes on the $\Mbar_{g,n}^\dagger$ spaces are eventually compatible under {\it cohomological pullback} and therefore defines an element in the logarithmic Chow ring~\cite{HMPW}. Since the survey has not appeared, we say a word about why the result is true; another version appears in~\cite[Section~4]{MR21}. 
	
	By toroidal semistable reduction, we can ensure that the forgetful map factorizes as
	\[
	\mathsf M_\Uplambda(X)^\dagger\to \mathsf M_\Uplambda(\mathsf A_X)^\dagger\to \Mbar_{g,n}^\dagger
	\]
	such that the second map is flat; the first map is strict and virtually smooth~\cite[Section~3]{R19}. For a further smooth subdivision $\Mbar_{g,n}^\ddagger\to \Mbar_{g,n}^\dagger$, this flatness ensures the fine and saturated fiber product
	\[
	\begin{tikzcd}
	\mathsf M_\Uplambda(\mathsf A_X)^\ddagger\arrow{d}\arrow{r} \arrow[rd, phantom, "\square"]&  \mathsf M_\Uplambda(\mathsf A_X)^\dagger\arrow{d} \\
	\Mbar_{g,n}^\ddagger \arrow{r} & \Mbar_{g,n}^\dagger
	\end{tikzcd}
	\]
	is a fiber product in ordinary stacks. The refined pullback for the horizontal arrows therefore coincide. By a diagram chase, compatibility of refined pullback and virtual pullback ensures the virtual classes are related by Gysin pullback along $\Mbar_{g,n}^\ddagger\to \Mbar_{g,n}^\dagger$. The assertion is a consequence. 
	
	\subsection{Comparing virtual structures} We compare the virtual classes of $\mathsf{M}_{\Uplambda}(X)$ and $\mathsf{R}_{\Uplambda}(\mathbb G_{\mathsf{log}}^r)$. Morally, the morphism from rigid to rubber is logarithmically smooth, so classes should be related by pullback, at least once subdivisions have been made to make the morphism flat. This is the content of Theorem~\ref{thm: virtpull}, whose statement and proof are our next tasks. We begin by recalling the obstruction theory in each case. 
	
	\subsubsection{The rigid geometry}We have a strict morphism of subdivisions of logarithmic mapping stacks:
	\[
	\mathsf{M}_{\Uplambda}(X)\to \mathsf M_\Uplambda(\mathsf A_X),
	\]
	where in the latter mapping stack we have imposed the stability condition of Section~\ref{sec: maps-to-tori}. 
	
	Abramovich--Wise equip the map above with a relative perfect obstruction theory. If $\pi\colon C\to \mathsf{M}_{\Uplambda}(X)$ is the universal curve and $f\colon C\to X$ is the universal map, the domain of the obstruction theory is given by $(R\pi_\star f^\star T^{\mathsf{log}}_X)^\vee$, see~\cite[Section~6]{AW}. The vector bundle $f^\star T^{\mathsf{log}}_X$ is isomorphic to $\mathcal O_C^{\oplus r}$, since the logarithmic tangent bundle of $X$ is trivial. Note that the substack $\mathsf M_\Uplambda(\mathsf A_X)\subset \mathfrak M_\Uplambda(\mathsf A_X)$ from Section~\ref{sec: maps-to-tvs} is open, so either can be used as the base of this obstruction theory. 
	
	\subsubsection{A variant of the rigid geometry}\label{sec: rigid-variant}  We have noted that $\mathsf M_\Uplambda(\mathbb G_{\mathsf{log}}^r)$ is typically not schematically representable. One way to produce representable subdivisions is to consider the $\mathsf{M}_{\Uplambda}(X)$ for a toric variety $X$. We will come across a more general source of representable subdivisions. If $B\to \mathsf M_\Uplambda(\mathbb G_{\mathsf{trop}}^r)$ is any subdivision by an algebraic stack with logarithmic structure, the base change along $\mathsf M_\Uplambda(\mathbb G_{\mathsf{log}}^r)\to \mathsf M_\Uplambda(\mathbb G_{\mathsf{trop}}^r)$ yields
	\[
	d\colon \mathsf M\to B.
	\]
	By Proposition~\ref{prop: strictness} this is schematically representable and strict. The stack $B$ need not be a subdivision of $\mathsf M^{\sf min}_\Uplambda(\mathsf A_X)$, so the spaces $\mathsf M$ are more general than the mapping stacks $\mathsf{M}_{\Uplambda}(X)$ above.
	
	Nevertheless, the morphism $d$ is equipped with a canonical perfect obstruction theory, with obstruction bundle $(R\pi_\star \mathcal O_C^{\oplus r})^\vee$. Roughly, this is because the map from $\mathsf M_\Uplambda(\mathbb G^r_{\mathsf{log}})\to \mathsf M_\Uplambda(\mathbb G^r_{\mathsf{trop}})$ is representable with obstruction theory given by the standard pull/push construction associated to the logarithmic tangent bundle of $\mathbb G_{\mathsf{log}}^r$, which is trivial. 
	
	More formally, one can copy the analysis of~\cite[Section~6]{AW}.  We provide a sketch. Consider the lifting problem associated to $d$. Let $S\to \mathsf M$ be a map of logarithmic schemes. Let $S\subset S'$ be a strict square zero extension with ideal $J$, and $S'\to B$ a compatible map:
	\[ 
	\vcenter{\xymatrix{
			& & \mathbb G_{\mathsf{log}}^r \ar[d]  \\
			C \ar[r] \ar@/^15pt/[urr]^f \ar[d] & C' \ar[r] \ar@{-->}[ur] \ar[d] & \mathbb G_{\mathsf{trop}}^r \\
			S \ar[r] & S'. &
	}} 
	\]
	Lifts in this diagram are controlled by the relative tangent bundle of $\mathbb G_{\mathsf{log}}^r\to \mathbb G_{\mathsf{trop}}^r$; note that although source and target are not schematically representable, the morphism is the stack quotient by $\mathbb G_m^r$, so it makes sense to speak of its relative tangent bundle (or at least the pullback to a scheme) and it is a rank $r$ trivial vector bundle.
	
	It follows that lifts of the diagram form a torsor on $C$ under the sheaf of abelian groups $\mathcal O_C^{\oplus r}\otimes J$, and we obtain an obstruction theory in the sense of Wise~\cite{Wis11}. The arguments of~\cite[Section~3]{R19} and of~\cite{AW} carry over verbatim to show the following compatibility. 
	
	\begin{lemma}\label{lem: extra-compatibility}
		Let $\mathsf M_\Uplambda(X)$ be (a subdivision of) the moduli space of logarithmic stable maps to $X$, and assume the morphism to $\mathsf M_\Uplambda(\mathbb G_{\mathsf{log}}^r)$ factors through $\mathsf M$ to give
		\[
		h\colon \mathsf M_\Uplambda(X)\to \mathsf M.
		\]
		The virtual classes of $\mathsf M_\Uplambda(X)$ and $\mathsf M$ are identified by pushforward along $h$. 
	\end{lemma}

	\subsubsection{The rubber geometry}  We move on to the rubber theory. Recall that $\mathsf R_\Uplambda(\mathbb G^r_{\mathsf{log}})$ can be understood as an intersection problem involving the Picard stack, after Marcus--Wise~\cite{MW17}. We describe the rank $1$ case for notational brevity, and indicate the required changes. In our notation, rubber moduli space is defined as the fiber product:
	\[
	\begin{tikzcd}
	\mathsf{R}_\Uplambda(\mathbb G_{\mathsf{log}})\arrow{d}\arrow{r} \arrow[rd, phantom, "\square"]&\mathsf{R}_\Uplambda(\mathbb G_{\mathsf{trop}})\arrow{d}{aj}\\
	\mathfrak M_{g,n}^{\mathsf{log}}\arrow[swap]{r}{0} & \mathbf{Pic}.
	\end{tikzcd}
	\]
	Note that $\mathsf{R}_\Uplambda(\mathbb G_{\mathsf{trop}})$ is (any subdivision of) the space $\mathbf{Div}$ in the work of Marcus--Wise~\cite{MW17} and is denoted $t\mathsf{DR}$ in~\cite{MR21}\footnote{A minor difference appears in that the balancing condition is not imposed in~\cite{MW17}. But as in~\cite{MR21}, the balancing condition must be satisfied in order for the Abel--Jacobi section to meet the $0$-section, so this difference is immaterial.}. The morphism $aj$ is the Abel--Jacobi section to the universal Picard variety, determined by the contact order $\Uplambda$. The bottom map is the $0$ section. It is a section of a smooth fibration, and therefore a regular embedding; this equips the top map also with an obstruction theory: explicitly, it gives rise to an embedding of the normal cone of the top arrow into the pullback of the normal bundle to $0$. 
	
	The normal bundle to the map $0$ at a curve $C$ is $H^1(C,\mathcal O_C)$: the obstruction to deforming $[C]$ as an object in the image of the $0$ section is precisely a deformation of the trivial bundle on $C$, i.e. an element of the tangent space to the identity in the Picard variety of $C$. Globally, the obstruction bundle is $(R^1\pi_\star \mathcal O_C)^\vee$. In the higher rank case, the bottom right in the square is replaced by the $r^{\rm th}$ fibered power of Picard varieties over $\fM_{g,n}^{\mathsf{log}}$, the $0$ map is replaced by the factorwise $0$ map, and the $aj$ map is replaced by the factorwise $aj$ map given by $\underline A$. The obstruction bundle is $(R^1\pi_\star \mathcal O_C^{\oplus r})^\vee$. %When considering rubber maps to a toric variety $X$, if $\pi:C\to \mathsf{R}_{\Uplambda}(X)$ is the universal curve, the obstruction theory is given by $(R^1\pi_\star \mathcal O_C\otimes N)^\vee$. 
	
	\subsubsection{Global charts and Siebert's formula}\label{sec: siebert-formula} In what follows, we compare virtual classes obtained by two constructions. As our setup is rather concrete, we can do so by means of traditional intersection theory, rather than chasing obstruction theories. Specifically, we use an elegant formula for the virtual class in terms of the Chern--Fulton class and the Segre class of the obstruction bundle, due to Siebert~\cite{Sie04}. We were led to the argument while puzzling over the related papers~\cite{AMS21,Tho22}. 
	
	We recall Siebert's idea. Let $K\to V$ be a morphism with a perfect obstruction theory $\mathbb E^\bullet$; assume $V$ is pure dimensional. Suppose $K$ can be embedded in a space $G$ which is smooth over $V$. In this case, we make sense of the Chern--Fulton class $c_F(K/V)$ of $K$ relative to $V$, by using the Chern class of the relative tangent bundle of $G$ over $V$ and the Segre class of $K$ in $G$. This is a Chow homology class that agrees with the total Chern class of the relative tangent bundle when the latter is defined. For further details, see~\cite[Example~4.2.6]{Ful98} and~\cite[Section~1]{Sie04}. 
	
	Siebert proves that the virtual fundamental class of $K$, which is by definition the virtual pullback of the fundamental class $[V]$, is given by
	\[
	[K]^{\mathsf{vir}} = \left\{ s(\mathbb E^{\bullet})\cap c_F(K/V)\right\}_{\mathsf{exp}},
	\]
	where the first term is the Segre class of the virtual vector bundle of deformations and obstructions. The final subscript takes the expected dimensional component. 
	
	To apply this, in addition to the obstruction bundles above, we need to know that our moduli spaces are embeddable in something smooth over the base of their obstruction theory. This is required for the existence of the Chern--Fulton class; see~\cite[Introduction]{Sie04}.
	
	In the rubber geometry, we have $\mathsf{R}_\Uplambda(\mathbb G^r_{\mathsf{log}})\to\mathsf{R}_\Uplambda(\mathbb G^r_{\mathsf{trop}})$ playing the role of the map $K\to V$ above. We are prepared to replace spaces by subdivisions, so the target of this morphism can be chosen to be smooth. Siebert's formula gives us control of the virtual class. Note that in this case, the definition of the Chern--Fulton class of $K$ over $V$ simplifies since the map in question is an embedding. If $K\to V$ is an embedding, then
	\[
	c_F(K/V) = c(T_V)^{-1}\cap c_F(K),
	\]
	where $c_F(K)$ is the usual Chern--Fulton class of the scheme $K$, see~\cite[Example~4.2.6]{Ful98}.
	
	In the case of the rigid geometry, we look at 
	\[
	\mathsf M_\Uplambda(X)\to \mathsf M_\Uplambda(\mathsf A_X).
	\]
	The target can be chosen to be smooth. The map is not an embedding, but as we will see in the course of the proof of Theorem~\ref{thm: virtpull}, it is a torus torsor over a closed substack of $\mathsf M_\Uplambda(\mathsf A_X)$. The torsor is pulled back from $\mathsf M_\Uplambda(\mathsf A_X)$ so again, we can use Siebert's formula. 
	
	\subsubsection{Comparison} We state and prove the rigid/rubber comparison. Recall there is a natural rigid-to-rubber morphism between minimal models:
	\[
	\epsilon: \mathsf M^{\mathsf{min}}_\Uplambda(\mathbb G^r_{\mathsf{log}})\to \mathsf{R}^{\mathsf{min}}_\Uplambda(\mathbb G^r_{\mathsf{log}}).
	\]
	The morphism is not schematically representable, as it is a $\mathbb G^r_{\mathsf{log}}$-torsor. Let 
	\[
	\mu\colon\mathsf{M}\to \mathsf{R}_\Uplambda(\mathbb G_{\mathsf{log}}^r)
	\]
	be a morphism of subdivisions of these spaces, with the map induced by $\epsilon$. The term ``induced'' means that the morphism is obtained by first subdividing the tropical space $\mathsf{R}^{\mathsf{min}}_\Uplambda(\mathbb G^r_{\mathsf{trop}})$, pulling back this subdivision to $\mathsf M^{\mathsf{min}}_\Uplambda(\mathbb G^r_{\mathsf{log}})$, and then subdividing further. 
	
	In order to state the comparison, we impose that $\mathsf M$ is schematically representable and that $\mu$ is flat; since $\mu$ is a subdivision of a logarithmic torus torsor, it is logarithmically smooth, and flatness can always be guaranteed by appropriate subdivision~\cite{AK00,Mol16}. We show the following:
	
	\begin{theorem}\label{thm: virtpull}
		For a subdivision $\mu$ of $\mathsf{M}\to \mathsf{R}_\Uplambda(\mathbb G_{\mathsf{log}}^r)$ as above, we have an equality
		\begin{equation*}
		[\mathsf{M}]^{\mathsf{vir}}= \mu^\star\left([\mathsf{R}_{\Uplambda}(\mathbb G_{\mathsf{log}}^r)]^{\mathsf{vir}}\right).
		\end{equation*}
		in the Chow homology of $\mathsf M$ in homological degree $3g-3+n-r(g-1)$, where $n$ is the number of marked points in $\Uplambda$.
	\end{theorem}
	
	\begin{proof}
		We have the following diagram, which is commutative but not cartesian:
		\[
		\begin{tikzcd}
		\mathsf M\arrow{d}\arrow{r} & B\arrow{d} \\
		\mathsf{R}_\Uplambda(\mathbb G_{\mathsf{log}}^r)\arrow{r} & \mathsf{R}_\Uplambda(\mathbb G_{\mathsf{trop}}^r).
		\end{tikzcd}
		\]
		Here, $B$ is a subdivision of the space $\mathsf M^{\mathsf{min}}_\Uplambda(\mathbb G^r_{\mathsf{trop}})$, and $\mathsf M$ is defined by $B\times_{\mathsf{R}^{\mathsf{min}}_\Uplambda(\mathbb G_{\mathsf{trop}}^r)} \mathsf{M}^{\mathsf{min}}_\Uplambda(\mathbb G_{\mathsf{trop}}^r)$. Therefore $\mathsf M\to B$ is strict. The symbols $\mathsf M$ and $B$ reprise their roles from Section~\ref{sec: rigid-variant}. %The subdivisions are chosen such that
		%\[
		%B\to \mathsf{R}_\Uplambda(\mathbb G_{\mathsf{trop}}^r)
		%\]
		%is flat. The horizontal morphisms are strict. %The symbols $\mathsf M$ and $B$ appearing here play the same role as in Section~\ref{sec: rigid-variant}.
		
		\noindent
		{\sc Step I. Observations about arrows.} The right vertical is a subdivision of the $\mathbb G_{\mathsf{trop}}^r$-torsor 
		\[
		\mathsf M^{\mathsf{min}}_\Uplambda(\mathbb G^r_{\mathsf{trop}})\to \mathsf{R}^{\mathsf{min}}_\Uplambda(\mathbb G_{\mathsf{trop}}^r).
		\]
		Thus the map $B\to\mathsf R_\Uplambda(\mathbb G_{\trop}^r)$ is logarithmically \'etale. The subdivisions in the setup are chosen such that the map
		\[
		B\to \mathsf{R}_\Uplambda(\mathbb G_{\mathsf{trop}}^r)
		\]
		is flat. The left vertical arrow is a subdivision of a $\mathbb G_{\mathsf{log}}^r$-torsor, induced by the corresponding subdivision of the $\mathbb G_{\mathsf{trop}}^r$-torsor $B\to \mathsf{R}_\Uplambda(\mathbb G_{\mathsf{trop}}^r)$. Thus, it is logarithmically smooth and flat of relative dimension $r$. The difference in dimensions illustrates the failure of the square to be cartesian. 
		
		The lower horizontal is a closed immersion: it is the inclusion of the intersections of the double ramification spaces. The upper horizontal in the commutative square is not a closed immersion. While the square is not cartesian, we can consider the fiber product $F$ instead, and there is a map:
		\[
		\mathsf M\to F.
		\]
		The morphism from $\mathbb G_{\mathsf{log}}\to \mathbb G_{\mathsf{trop}}$ is a quotient by $\mathbb G_m$, so it is smooth. Therefore $\mathsf M\to F$ is also a torus torsor, and therefore is also smooth. 
		
		It will be useful to describe the morphism $\mathsf M\to F$ more explicitly; we set $r= 1$ for simplicity. The map $\mathsf M\to F$ strict, and it can be understood via the fiber product description of $F$:
		\[
		F = \mathsf{R}_\Uplambda(\mathbb G_{\log}) \times_{\mathsf{R}_\Uplambda(\mathbb G_{\trop})} B.
		\]
		A point of $B$ determines a curve together with global section of its characteristic monoid, i.e. an element $\overline \alpha$ in $H^0(C,\overline M_C^{\mathsf{gp}})$. Associated to this is a $\mathcal O_C^\star$-torsor of lifts of $\overline\alpha$ to $M_C^{\mathsf{gp}}$, which we denote $\mathcal O_C(-\overline\alpha)$, see for example~\cite[Section~2.6]{RSW17B}; the sign convention is immaterial but we choose it for consistency with the formula given for it in loc. cit. The fiber product imposes the condition that this torsor is isomorphic to the trivial torsor. The space $\mathsf M$ has, in addition to these data, a {\it distinguished} trivialization of the torsor $\mathcal O_C(-\overline\alpha)$. We can therefore view the $\mathbb G_m$-torsor $\mathsf M\to F$, as the torsor associated to the line bundle $R^0\pi_\star\mathcal O_C$, where as before $\pi:C\to \mathsf M$ is the universal curve. More canonically, this torsor can be viewed as the torsor of isomorphisms $\mathsf{Iso}(\mathcal O_C(-\overline\alpha),\mathcal O_C)$. For higher values of $r$, one obtains $r$ copies of this torsor. 
		
		We make a final note before proceeding. The space $\mathsf M$ is proper, and yet it is a torus torsor over the space $F$. While this might be slightly disorienting at first, the reader may note that $F\to\mathsf{R}_\Uplambda(\mathbb G_{\log}^r)$ is relatively $0$-dimensional, but of Artin type. A proper space can certainly be exhibited as a torus torsor over an Artin stack, e.g. consider $\PP^1\to[\PP^1/\mathbb G_m]$.
		
		\noindent
		{\sc Step II. Obstruction theories.} We compare the virtual classes on these spaces. Consider the closed immersion $F\to B$. We can view it as a fiber product in two different ways:
		\[
		\begin{tikzcd}
		F\arrow{d}\arrow{r}\arrow[rd, phantom, "\square"] & B\arrow{d} \\
		\mathsf{R}_\Uplambda(\mathbb G_{\mathsf{log}}^r)\arrow{r} \arrow{d}\arrow[rd, phantom, "\square"]& \mathsf{R}_\Uplambda(\mathbb G_{\mathsf{trop}}^r)\arrow{d}{aj} \\
		\mathfrak M_{g,n}^{\mathsf{log}}\arrow{r}{0} & \prod_{i=1}^r \mathbf{Pic},
		\end{tikzcd}
		\]
		where $\prod_{i=1}^r \mathbf{Pic}$ is the product, over $\mathfrak M_{g,n}^{\mathsf{log}}$ of $r$ copies of the Picard scheme. The $0$ denotes the coordinatewise trivial bundle. 
		Now both squares are cartesian. The embedding of the $0$ section has normal bundle $H^1(C,\mathcal O_C)^{\oplus r}$ at the curve $C$ and forms a rank $g$ vector bundle. Pullback equips the middle and top arrow with a perfect obstruction theory, i.e. an embedding of the cones of these horizontal morphisms into this vector bundle. Since the right vertical map on the top square is flat, the two horizontals in the top square have compatible perfect obstruction theories. If $\pi:C\to F$ is the universal curve, the obstruction bundle is $R^1\pi_\star(\mathcal O_C^{\oplus r})^\vee$, see~\cite[Proposition~5.3.6.2]{MW17}. 
		
		The morphism $\mathsf M\to F$ is smooth, and we have equipped $F$ with a perfect obstruction theory over $B$, we have two ways of endowing $\mathsf M$ with a virtual class. First, we can perform the standard virtual pullback of the fundamental class along the morphism 
		\[
		\mathsf M\to B,
		\]
		using the obstruction theory defined earlier in this section. Second, we can use the obstruction theory on $F\to B$ to obtain a virtual class $[F]^{\mathsf{vir}}$ and then perform flat pullback for the class $[F]^{\mathsf{vir}}$ along $\mathsf M\to F$. We claim that these homology classes on $\mathsf M$ coincide. 
		
		The claim follows from Siebert's formula, discussed in Section~\ref{sec: siebert-formula}. Indeed, $\mathsf M$ is a torus bundle over a closed substack $F$ inside $B$, and as we have described above, this torus torsor is isomorphic to the one associated to the bundle $R^0\pi_\star(\mathcal O_C\otimes N)$, which is certainly pulled back from $B$. We are free to use Siebert's formula, and writing out the classes yields the result. To spell out the details, recall that the formula requires the Chern--Fulton class and the obstruction bundle, as a virtual vector bundle. The obstruction bundle on $F$ is $R^1\pi_\star(\mathcal O_C\otimes N)$, while the obstruction bundle on $\mathsf M$ relative to $B$ is $R^1\pi_\star(\mathcal O_C\otimes N)-R^0\pi_\star(\mathcal O_C\otimes N)$. Observe the $R^0$ term is a trivial vector bundle, so the total Segre classes of these two virtual vector bundles coincide, and their ranks differ by $r$. 
		
		The relative tangent bundle of $\mathsf M\to F$ is trivial, since it is a torus torsor. Thus, the Fulton--Chern class of $\mathsf M$ is equal to the pullback of that of $F$. The result follows from Siebert's formula. 
	\end{proof}

	We now prove Theorem~\ref{thm:B}. The construction of Proposition~\ref{prop:toriccart} will be used below. %Specifically we are now able to transfer the virtual classes of our moduli spaces of maps to $X$ to the spaces $X_1$ and $X_2$ via proper pushforward. 

	\subsection{Proof of Theorem~\ref{thm:B}} 
	There are three steps in the proof, starting from an evaluation class operating on the class $[\mathsf M_\Uplambda(X)]^{\mathsf{vir}}$. First, we express the pushforward of this to the moduli space of curves as an analogous pairing on the rubber space. Next, we replace the rubber cohomology class with a piecewise polynomial. Finally, we move this to a piecewise polynomial class over the higher double ramification cycle on the moduli space of curves. 
	
	\noindent
	{\sc Step I. rigid-to-rubber.} We start with the diagram of rubber, rigid, and evaluation spaces: 
	\[
	\begin{tikzcd}
	\mathsf{M}^{\mathsf{min}}_{\Uplambda}(\mathbb{G}_{\mathsf{log}}^r) \arrow[r, ] \arrow[d,swap,"\varepsilon"] \arrow[rd, phantom, "\square"]
	& \mathsf{Ev}_{\Uplambda}(\mathbb{G}_{\mathsf{log}}^r) \arrow[d,] \\
	\mathsf{R}^{\mathsf{min}}_\Uplambda(\mathbb{G}_{\mathsf{log}}^r) \arrow[r,]
	&  \mathsf{Ev}_{\Uplambda}^{\mathsf{rub}}(\mathbb{G}_{\mathsf{log}}^r).
	\end{tikzcd}
	\]
	Fix a toric variety $X$ and discrete data $\Uplambda$. Choose a cohomology class $\gamma$ in some subdivision of $\mathsf{Ev}_\Uplambda(X)$. Noting our convention is to suppress subdivisions, by Proposition~\ref{prop:toriccart} we can choose a set of subdivisions to get a diagram:
	\[
	\begin{tikzcd}
	\mathsf{M}_{\Uplambda}(X) \arrow[rd, "\varepsilon"] \arrow[r, "p"] & \mathsf M \arrow[r, "\mathsf{ev}"] \arrow[d,swap,"\mu"] \arrow[rd, phantom, "\square"]
	& \mathsf{Ev}_{\Uplambda}(X) \arrow[d,"\delta"] \\
	& \mathsf{R}_\Uplambda(\mathbb G_{\mathsf{log}}^r) \arrow[r,swap,"\mathsf{ev}_{\mathsf{rub}}"]
	&  \mathsf{Ev}_{\Uplambda}^{\mathsf{rub}}(X),
	\end{tikzcd}
	\]
	with the property that the vertical arrows are flat. The initial choice of $\mathsf{Ev}_\Uplambda(X)$ may have to be refined. If this is done $\gamma$ is pulled back. The morphism $p$ is a subdivision. 
	%The pullback compatibility of Proposition~\ref{prop:virtpull} still holds for the pushforwarded virtual class on $\mathsf M$ and $\tilde{\varepsilon}$.
	
	Now observe that $\mathsf M$ is precisely one of the spaces considered in Section~\ref{sec: rigid-variant}. It is therefore equipped with a natural virtual fundamental class $[\mathsf M]^{\mathsf{vir}}$. By Lemma~\ref{lem: extra-compatibility}, we have
	\[
	[\mathsf M]^{\mathsf{vir}} = p_\star [\mathsf M_\Uplambda(X)]^{\mathsf{vir}}. 
	\]
	By this compatibility, and the projection formula applied to $p$ and the cohomology class $\mathsf{ev}^\star(\gamma)$, we see our Gromov--Witten cycle is equal to the pushforward of
	\[
	[\mathsf M]^{\mathsf{vir}}\cap \mathsf{ev}^\star(\gamma)
	\]  
	to the moduli space of curves. We now discard the triangle on the left side of the diagram. 
	
	By Theorem~\ref{thm: virtpull}, we also have an equality
	\[
	[\mathsf M]^{\mathsf{vir}}=\mu^\star[\mathsf{R}_\Uplambda(\mathbb G_{\mathsf{log}}^r)]^\mathsf{vir}.
	\]
	The morphism from $\mathsf M$ to $\Mbar_{g,n}$ factors through $\mathsf{R}_\Uplambda(\mathbb G_{\mathsf{log}}^r)$, so we can apply the projection formula, this time for the morphism $\mu$:
	\[
	\mu_\star\left(\mathsf{ev}^\star(\gamma)\cap \mu^\star[\mathsf{R}_\Uplambda(\mathbb G_{\mathsf{log}}^r)]^\mathsf{vir}\right) = \mu_\star(\mathsf{ev}^\star(\gamma))\cap [\mathsf{R}_\Uplambda(\mathbb G_{\mathsf{log}}^r)]^{\mathsf{vir}}.
	\]
	It remains to calculate $\mu_\star(\mathsf{ev}^\star(\gamma))$ in more concrete terms. The spaces in the square are all proper so the vertical morphisms are as well, and we may pushforward cohomology classes. By compatibility of pushforward and pullback in the square, we obtain:
	\[
	\mu_\star(\mathsf{ev}^\star(\gamma)) = \mathsf{ev}_{\mathsf{rub}}^\star\delta_\star(\gamma).
	\]
	We have therefore expressed the pushforward along $\varepsilon$ of the logarithmic Gromov--Witten cycle on $\mathsf M_\Uplambda(X)$ as the integral of the virtual class on the rubber moduli space with $\mathsf{ev}^\star_{\mathsf{rub}}\delta_\star(\gamma)$. 
	
	\noindent
	{\sc Step II. rubber-to-Artin-fan.} By a further refinement of the subdivision, we are free to assume that $\mathsf{Ev}_{\Uplambda}^{\mathsf{rub}}(X)$ is smooth. We make such a subdivision and reset the notation for this space to $\mathsf E$ for brevity. Let $H$ be its dense torus. Since $\mathsf E$ is smooth, the pullback map on Chow cohomology
	\[
	\mathsf{CH}^\star([\mathsf E/H])\to \mathsf{CH}^\star(E)
	\]
	is surjective~\cite[Section~2.3]{Bri97}. The quotient stack $[\mathsf E/H]$ is the Artin fan $\mathsf A_{\mathsf E}$. Let $\gamma'$ be a lift of $\delta_\star\gamma$ to the Chow ring of $\mathsf A_E$, so it can be interpreted as a piecewise polynomial function on the fan of $E$. 
	
	By the discussion of~\cite[Section~4]{MR21} we have an embedding map
	\[
	\mathsf{R}_{\Uplambda}(\mathbb G_{\mathsf{log}}^r)\hookrightarrow \mathsf{R}_{\Uplambda}(\mathbb G_{\mathsf{trop}}^r).
	\]
	Recall the codomain of the map above is obtained from $\Mbar_{g,n}$ by performing a subdivision (possibly including root constructions) and passing to the complement of some subset of closed strata. The rubber evaluation can be composed with the map to $\mathsf A_E$:
	\[
	\mathsf{R}_\Uplambda(\mathbb G_{\mathsf{log}}^r)\to \mathsf A_{\mathsf E}.
	\]
	The map certainly factors through $\mathsf{R}_\Uplambda(\mathbb G_{\mathsf{trop}}^r)$. We can therefore understand our Gromov--Witten cycle as a refined intersection of the rubber virtual class with the piecewise polynomial $\gamma'$.
	
	\noindent
	{\sc Step III. Artin-fan-to-moduli-of-curves.} The morphism
	\[
	\mathsf{R}_\Uplambda(\mathbb G_{\mathsf{trop}}^r)\to \mathsf A_{\mathsf E}
	\]
	is logarithmic, and therefore equivalent to a map on the underlying generalized cone complexes~\cite[Proposition~2.10]{ACGS15}. The generalized cone complex of $\mathsf{R}_\Uplambda(\mathbb G_{\mathsf{trop}}^r)$ is a subcomplex of a subdivision of the cone complex of $\Mbar_{g,n}$. Choose this latter subdivision to be smooth, and denote it by $\Sigma$. 
	
	The morphism from $\mathsf{R}_\Uplambda(\mathbb G_{\mathsf{trop}}^r)$ to $\mathsf A_{\mathsf E}$ is specified by a set of piecewise linear functions on some subset of rays in $\Sigma$, i.e. those rays that lie in the subcomplex that defines $\mathsf{R}_\Uplambda(\mathbb G_{\mathsf{trop}}^r)$ as an open. By choosing the values on the remaining rays (arbitrarily) and then reversing the dictionary, we obtain a morphism from a subdivision of $\Mbar_{g,n}$ to $\mathsf A_{\mathsf E}$. Summarizing, this blowup which we denote $\overline{\mathcal M}^\dagger$, contains $\mathsf{R}_\Uplambda(\mathbb G_{\mathsf{trop}}^r)$ and therefore $\mathsf{R}_\Uplambda(\mathbb G_{\mathsf{log}}^r)$ as a substack, and upon restriction to these substacks, we recover the evaluation maps above. 
	
	Apply the projection formula for the pushforward from $\mathsf{R}_\Uplambda(\mathbb G_{\mathsf{log}}^r)$ into the blowup $\overline{\mathcal M}^\dagger$ with the class $\gamma'$. We have expressed the required Gromov--Witten cycle as a class in the logarithmic Chow ring of $\Mbar_{g,n}$. By applying~\cite{HMPPS,HS21,MR21}, the pushforward of $\mathsf{R}_\Uplambda(\mathbb G_{\mathsf{log}}^r)$ is tautological in the logarithmic Chow ring. The result follows.
	\qed

	\subsection{Classical connections}\label{sec: severi-degrees} We conclude with a discussion of two practical examples: {\it Hurwitz numbers} of $\mathbb P^1$ and {\it Severi degrees} of $\mathbb P^2$. We begin with the latter, since the former are treated in some detail in~\cite[Section~3]{CMR22}. Further links with tropical curve counting may be found in~\cite{KHSUK}. 
	
	\subsubsection{Severi degrees} The Severi degrees are invariants from classical enumerative geometry, and are recalled below. They are also the ordinary Gromov--Witten invariants of $\mathbb P^2$ with point insertions. We connect the Severi degrees of $\mathbb P^2$ to rank $2$ higher double ramification cycles. 
	
	The Severi degrees concern the enumeration of genus $g$ curves of degree $d$ in $\mathbb{P}^2$, through $3d-1+g$ general points. The term {\it Severi degree} comes from the equality with the degree of the variety parameterizing genus $g$ curves in the family of degree $d$ plane curves, namely the {\it Severi variety}.
	
	The Severi degrees can be computed by tropical geometry, via Mikhalkin's correspondence theorem~\cite{Mi03}, though we think of this via its logarithmic geometry incarnation~\cite{MR16,NS06}. The number $N_{d, g}$ is also equal to $(d!)^3$ times a {\it logarithmic} Gromov--Witten invariant of $\mathbb{P}^2$ with its toric log structure\footnote{The implicit equality of logarithmic and ordinary Gromov--Witten invariants of $\mathbb P^2$ here is a happy accident; it holds only for point insertions, and already fails for non-Fano toric surfaces. The logarithmic invariants of a toric surface are always equal to toric Severi degrees but this typically fails for the ordinary Gromov--Witten invariants of Hirzebruch surfaces~\cite{R15-Severi}.}. The discrete data $\Uplambda$ is as follows: the genus is $g$ and the contact order matrix is:
	\[
	\underline{A}=
	\begin{pmatrix}
	1 \cdots 1 & 0 \cdots 0 & -1 \cdots -1 & 0 \cdots 0\\
	0 \cdots 0 & 1 \cdots 1 & -1 \cdots -1 & 0 \cdots 0
	\end{pmatrix}.
	\]
	The matrix has $3d$ nonzero columns, written in $3$ groups of $d$ each. The number of zero columns is equal to $3d-1+g$. We obtain a logarithmic Gromov--Witten invariant by pulling back a point constraint corresponding to each of these $3d-1+g$ points. 
	
	Our goal is to use the method in the proof of Theorem~\ref{thm:B} to provide a piecewise polynomial class in the logarithmic Chow ring of the moduli space of curves, which when intersected with the logarithmic toric contact cycle computes the Severi degree $N_{d,g}$. We recall the relevant diagram:
	\[
	\begin{tikzcd}
	\mathsf{M}_{\Uplambda}(\mathbb{P}^2) \arrow[rd, "\varepsilon"] \arrow[r, "p"] & \mathsf M \arrow[r, "\mathsf{ev}"] \arrow[d,swap,"\mu"] \arrow[rd, phantom, "\square"]
	& \mathsf{Ev}_{\Uplambda}(\mathbb{P}^2) \arrow[d,"\delta"] \\
	& \mathsf{R}_\Uplambda(\mathbb G_{\mathsf{log}}^2) \arrow[r,swap,"\mathsf{ev}_{\mathsf{rub}}"]
	&  \mathsf{Ev}_{\Uplambda}^{\mathsf{rub}}(\mathbb{P}^2),
	\end{tikzcd}
	\]
	By convention, we take the evaluation spaces to be a restricted one, corresponding only to the markings with contact order $0$. We set $n = 3d-g+1$ and take $\mathsf{Ev}_{\Uplambda}(\mathbb{P}^2)$ to be a subdivision of $(\mathbb{P}^2)^n$. Similarly $\mathsf{Ev}_{\Uplambda}^{\mathsf{rub}}(\mathbb{P}^2)$ will be a subdivision of a toric quotient of $(\mathbb{P}^2)^n$. 
	
	The cocharacter space of the torus in $\mathsf{Ev}_{\Uplambda}^{\mathsf{rub}}(\mathbb{P}^2)$ can be viewed as a moduli space of configurations of $n$ points in $\mathbb{R}^2$ up to translation. This furnishes a tropical interpretation of $\mathsf{ev}_{\mathsf{rub}}$. Recall that $\mathsf{ev}_{\mathsf{rub}}$ factors through $\mathsf{R}_\Uplambda(\mathbb G_{\mathsf{trop}}^r)\to \mathsf A_{\mathsf E}$.
	We denote the induced map from the generalized cone complex of $\mathsf{R}_\Uplambda(\mathbb G_{\mathsf{trop}}^r)$ to the fan of $\mathsf{Ev}_{\Uplambda}^{\mathsf{rub}}(\mathbb{P}^2)$ by $\mathsf{ev}_{\mathsf{rub}}^{\mathsf{trop}}$. We can view this generalized cone complex as a moduli space of tropical maps from tropical curves to $\mathbb{R}^2$ up to translation. The contact order $0$ markings correspond to infinite legs of the tropical curve which are contracted to points in $\mathbb{R}^2$. 
	
	Now $\mathsf{ev}_{\mathsf{rub}}^{\mathsf{trop}}$ sends a tropical curve to this configuration of points in $\mathbb{R}^2$ up to translation. We have a class $\gamma= (\mathsf{pt})^{\otimes n}$ in $\mathsf{CH}^\star(\mathsf{Ev}_{\Uplambda}(\mathbb{P}^2))$
	which is Poincar\'e dual to a point class. This implies that $\delta_\star(\gamma)$ is Poincar\'e dual to a point class. The choice of a maximal dimensional cone in the fan of $E$ gives us a piecewise polynomial $\phi$ which is a lift of  $\delta_\star(\gamma)$ to $\mathsf{CH}^\star(A_{\mathsf E})$. Since $\mathsf{ev}_{\mathsf{rub}}$ is logarithmic, the pullback of $\delta_\star(\gamma)$ to $\mathsf{R}_\Uplambda(\mathbb G_{\mathsf{trop}}^r)$ is given by the piecewise polynomial $\phi \circ \mathsf{ev}_{\mathsf{rub}}^{\mathsf{trop}}$.
	As in the proof of Theorem~\ref{thm:B}, choose a smooth subdivision $\overline{\mathcal M}^\dagger$ of $\Mbar_{g,n}$ from which $\phi \circ \mathsf{ev}_{\mathsf{rub}}^{\mathsf{trop}}$ is pulled back. By the projection formula, we can intersect the virtual class of $\mathsf{R}_\Uplambda(\mathbb{P}^2)$ with $\mathsf{ev}_{\mathsf{rub}}^\star(\delta_\star(\gamma))$ in the Chow ring of $\overline{\mathcal M}^\dagger$. The intersection number is the Severi degree.
	
	\subsubsection{The Severi degrees in practice} Let us sketch what the piecewise polynomial looks like in one of the simplest cases -- degree $2$ curves in $\mathbb P^2$ of genus $0$ through $5$ points. Of course, the number is $1$. In keeping with the discussion above, the integral can be expressed on $\Mbar_{0,11}$. Among the $11$ points, $6$ of them correspond to the boundary contact points in $\PP^2$ and the remaining $5$ are contact order $0$. 
	
	The insertion $\gamma$ is $(\mathsf{pt})^{\otimes 5}$ in $(\PP^2)^5$. The appropriate logarithmic cohomology class $\gamma_{\sf rub}$ is the image of a piecewise polynomial in logarithmic Chow, but we abuse notation slightly and let $\gamma_{\sf rub}$ denote the piecewise polynomial itself:
	\[
	\gamma_{\mathsf{rub}}\colon \cM_{0,11}^\trop\to \RR.
	\]
	It associates a real number to each $11$-pointed, genus $0$ tropical curve $\Gamma$. Such a tropical curve admits a unique balanced map to $\RR^2$ subject to the following conditions: (i) the marked point $p_7$ is sent to the origin, and (ii) the slopes along the first $6$ marked points are as given by the contact order. Note this uniqueness follows from e.g.~\cite{GKM07,RW19}. A typical $11$-pointed genus $0$ tropical curve can be found in Figure~\ref{fig: 11-marked-guy}.
	
\begin{figure}

\tikzset{every picture/.style={line width=0.75pt}} %set default line width to 0.75pt        

\begin{tikzpicture}[x=0.75pt,y=0.75pt,yscale=-1,xscale=1]
%uncomment if require: \path (0,300); %set diagram left start at 0, and has height of 300

%Straight Lines [id:da2567285421820087] 
\draw    (272,156) -- (314,124) ;
%Straight Lines [id:da06226278890258219] 
\draw    (314,124) -- (314,91) ;
%Straight Lines [id:da8986799330622635] 
\draw    (314,124) -- (352,124) ;
%Straight Lines [id:da001826538676565459] 
\draw    (352,124) -- (351.03,198) ;
\draw [shift={(351,200)}, rotate = 270.75] [color={rgb, 255:red, 0; green, 0; blue, 0 }  ][line width=0.75]    (10.93,-3.29) .. controls (6.95,-1.4) and (3.31,-0.3) .. (0,0) .. controls (3.31,0.3) and (6.95,1.4) .. (10.93,3.29)   ;
%Straight Lines [id:da009926496556084352] 
\draw    (352,124) -- (418.44,70.26) ;
\draw [shift={(420,69)}, rotate = 141.03] [color={rgb, 255:red, 0; green, 0; blue, 0 }  ][line width=0.75]    (10.93,-3.29) .. controls (6.95,-1.4) and (3.31,-0.3) .. (0,0) .. controls (3.31,0.3) and (6.95,1.4) .. (10.93,3.29)   ;
%Straight Lines [id:da9175848423436128] 
\draw    (314,91) -- (380.44,37.26) ;
\draw [shift={(382,36)}, rotate = 141.03] [color={rgb, 255:red, 0; green, 0; blue, 0 }  ][line width=0.75]    (10.93,-3.29) .. controls (6.95,-1.4) and (3.31,-0.3) .. (0,0) .. controls (3.31,0.3) and (6.95,1.4) .. (10.93,3.29)   ;
%Straight Lines [id:da9676958699997447] 
\draw    (272,156) -- (271.03,230) ;
\draw [shift={(271,232)}, rotate = 270.75] [color={rgb, 255:red, 0; green, 0; blue, 0 }  ][line width=0.75]    (10.93,-3.29) .. controls (6.95,-1.4) and (3.31,-0.3) .. (0,0) .. controls (3.31,0.3) and (6.95,1.4) .. (10.93,3.29)   ;
%Straight Lines [id:da6518062582343229] 
\draw    (272,156) -- (204,156) ;
\draw [shift={(202,156)}, rotate = 360] [color={rgb, 255:red, 0; green, 0; blue, 0 }  ][line width=0.75]    (10.93,-3.29) .. controls (6.95,-1.4) and (3.31,-0.3) .. (0,0) .. controls (3.31,0.3) and (6.95,1.4) .. (10.93,3.29)   ;
%Straight Lines [id:da7768635007430658] 
\draw    (314,91) -- (206,91) ;
\draw [shift={(204,91)}, rotate = 360] [color={rgb, 255:red, 0; green, 0; blue, 0 }  ][line width=0.75]    (10.93,-3.29) .. controls (6.95,-1.4) and (3.31,-0.3) .. (0,0) .. controls (3.31,0.3) and (6.95,1.4) .. (10.93,3.29)   ;
%Straight Lines [id:da545212696213125] 
\draw  [dash pattern={on 0.84pt off 2.51pt}]  (259,91) -- (232,60) ;
\draw [shift={(259,91)}, rotate = 228.95] [color={rgb, 255:red, 0; green, 0; blue, 0 }  ][fill={rgb, 255:red, 0; green, 0; blue, 0 }  ][line width=0.75]      (0, 0) circle [x radius= 3.35, y radius= 3.35]   ;
%Straight Lines [id:da17174440900931698] 
\draw  [dash pattern={on 0.84pt off 2.51pt}]  (348,63.5) -- (321,34) ;
\draw [shift={(348,63.5)}, rotate = 227.53] [color={rgb, 255:red, 0; green, 0; blue, 0 }  ][fill={rgb, 255:red, 0; green, 0; blue, 0 }  ][line width=0.75]      (0, 0) circle [x radius= 3.35, y radius= 3.35]   ;
%Straight Lines [id:da07167201222096664] 
\draw  [dash pattern={on 0.84pt off 2.51pt}]  (352,174) -- (327,149) ;
\draw [shift={(352,174)}, rotate = 225] [color={rgb, 255:red, 0; green, 0; blue, 0 }  ][fill={rgb, 255:red, 0; green, 0; blue, 0 }  ][line width=0.75]      (0, 0) circle [x radius= 3.35, y radius= 3.35]   ;
%Straight Lines [id:da11061487003100412] 
\draw  [dash pattern={on 0.84pt off 2.51pt}]  (237,156) -- (201,120) ;
\draw [shift={(237,156)}, rotate = 225] [color={rgb, 255:red, 0; green, 0; blue, 0 }  ][fill={rgb, 255:red, 0; green, 0; blue, 0 }  ][line width=0.75]      (0, 0) circle [x radius= 3.35, y radius= 3.35]   ;
%Straight Lines [id:da7619148244892358] 
\draw  [dash pattern={on 0.84pt off 2.51pt}]  (270,212) -- (246,188) ;
\draw [shift={(270,212)}, rotate = 225] [color={rgb, 255:red, 0; green, 0; blue, 0 }  ][fill={rgb, 255:red, 0; green, 0; blue, 0 }  ][line width=0.75]      (0, 0) circle [x radius= 3.35, y radius= 3.35]   ;

\end{tikzpicture}
\caption{An $11$-pointed genus $0$ tropical curve, drawn in $\mathbb R^2$ using the unique balanced map with the given contact orders. Note that the dotted lines correspond to the $5$ marked points of contact order $0$ while the $6$ ``ends'' with arrow heads correspond to the boundary markings. }\label{fig: 11-marked-guy}
\end{figure}
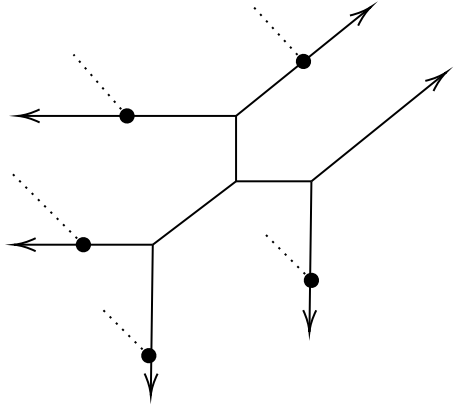
	
	For each of the marked points $p_8,\ldots, p_{11}$ we attach a real number; the product  of these will be the value of $\gamma_{\mathsf{rub}}$ on $\Gamma$. Each marked point is supported at a vertex of the tropical curve $\Gamma$. For each $j$ in $\{8\ldots,11\}$, there is a unique path in $\Gamma$ connecting the vertex supporting $p_7$ with $p_j$. This path maps piecewise linearly to $\RR^2$, and we can project onto the $x$-axis. Define the value $x_{\Gamma,j}$ to be the length of this image if it is contained in the positive axis, and $0$ if this path intersects the negative $x$-axis. Define $y_{\Gamma,j}$ similarly. The function $\gamma_{\sf rub}$ takes the value on $\Gamma$ to be $\prod_{j=8}^{11} x_{\Gamma,j}y_{\Gamma,j}$. 
	
	In this case, we have taken the rubber evaluation space to be $(\PP^2)^4$. We have chosen the piecewise polynomial that represents the equivariant class corresponding to the point $(1\colon 0\colon 0)$ in each factor and the class $\gamma_{\mathsf{rub}}$ is its pullback to a blowup of $\Mbar_{0,11}$ via the procedure above. It is combinatorially somewhat nontrivial to evaluate this piecewise polynomial and get the number $1$, but this is the subject of tropical intersection theory~\cite{GKM07}. We do not explain this here, but in this case, it is relatively easy to check that the piecewise polynomial is $0$ on most cones of $\cM^{\trop}_{0,11}$, and is nonzero on exactly one $8$-dimensional cone of a subdivision. The polynomial on this cone is equal to the product of the canonical ``slope $1$'' functions associated to the rays so the intersection number is $1$. We leave further details of the calculation to the reader.
	
	\subsubsection{Hurwitz theory} Intersections against the logarithmic double ramification cycle in rank $1$ also contain information related to classical algebraic geometry. It was recently shown that double Hurwitz numbers of $\mathbb P^1$ are intersections of the logarithmic double ramification cycle with piecewise polynomials~\cite{CMR22}. Note this includes the classical Hurwitz numbers counting covers of the Riemann sphere, with simple branching. We treat only this case for simplicity. The result is not, strictly speaking, a special case of our main result. However, it is close in spirit so we record it. 
	
	We take the target to be $\mathbb P^1$ equipped with its toric logarithmic structure, and fix a genus $g$ and degree $d$. The expected dimension of covers is $2d-2+2g$. The contact order vector is given by
	\[
	A = (1,\ldots,1,-1,\ldots,-1,0,\ldots 0),
	\]
	where $d$ entries are $1$, another $d$ are $-1$, and $2d-2+2g$ are $0$. The Hurwitz number is:
	\[
	H_g(d) = \frac{1}{(d!)^2} \int_{\mathsf M_{g,A}(\mathbb P^1)} \prod_{i = 1}^{2d-2+2g} \mathsf{ev}_i^\star(\mathsf{pt})\psi_i,
	\]
	where the insertions are placed at the contact order $0$ markings. The key point is that the class $\mathsf{ev}_i^\star(\mathsf{pt})\psi_i$ imposes a simple ramification condition at $p_i$. This is a codimension $1$ constraint, and since there are $2d-2+2g$ ramification points, the count is finite and equal to the Hurwitz number. 
	
	Let $A'$ be the vector of length $2d$ which is the vector $A$ with the final $0$ entries removed. Since the virtual class pulls back along the forgetful morphism
	\[
	\mathsf M_{g,A}(\mathbb P^1)\to\mathsf M_{g,A'}(\mathbb P^1)
	\]
	the Hurwitz number can be computed by an integral on $\mathsf M_{g,A'}(\mathbb P^1)$. By fixing the first ramification point to be at the point $1$ in $\mathbb P^1$, we can further reduce this to an integral on $\mathsf{R}_{g,A'}(\mathbb P^1)$. The main result of~\cite[Section~3]{CMR22} is an explicit piecewise polynomial class, called the {\it branch polynomial}, which when integrated against the double ramification class $\mathsf{logDR}_g(A')$ in the logarithmic Chow ring of $\Mbar_{g,2d}$ becomes equal to the Hurwitz number. 
	
	Note that in the paper above, the double ramification expression is used to construct analogues of the double Hurwitz numbers that are related to pluricanonical divisors in the same way that the ordinary numbers are related to principal divisors. The same could be done here to construct ``pluricanonical'' analogues of the Severi degrees. 
	\bibliographystyle{siam} 
	\bibliography{LogGWandDR} 

\begin{thebibliography}{10}

\bibitem{AMS21}
{\sc M.~Abouzaid, M.~McLean, and I.~Smith}, {\em {Complex cobordism,
  Hamiltonian loops and global Kuranishi charts}}, arXiv:2110.14320,  (2021).

\bibitem{ACP}
{\sc D.~Abramovich, L.~Caporaso, and S.~Payne}, {\em The tropicalization of the
  moduli space of curves}, Ann. Sci. {\'E}c. Norm. Sup{\'e}r., 48 (2015),
  pp.~765--809.

\bibitem{AC11}
{\sc D.~Abramovich and Q.~Chen}, {\em Stable logarithmic maps to
  {D}eligne-{F}altings pairs {II}}, Asian J. Math., 18 (2014), pp.~465--488.

\bibitem{ACGS15}
{\sc D.~Abramovich, Q.~Chen, M.~Gross, and B.~Siebert}, {\em Decomposition of
  degenerate {G}romov-{W}itten invariants}, Compos. Math., 156 (2020),
  pp.~2020--2075.

\bibitem{ACMUW}
{\sc D.~Abramovich, Q.~Chen, S.~Marcus, M.~Ulirsch, and J.~Wise}, {\em
  Skeletons and fans of logarithmic structures}, in Nonarchimedean and Tropical
  Geometry, M.~Baker and S.~Payne, eds., Simons Symposia, Springer, 2016,
  pp.~287--336.

\bibitem{ACMW}
{\sc D.~Abramovich, Q.~Chen, S.~Marcus, and J.~Wise}, {\em Boundedness of the
  space of stable logarithmic maps}, {J. Eur. Math. Soc.}, 19 (2017),
  pp.~2783--2809.

\bibitem{AK00}
{\sc D.~Abramovich and K.~Karu}, {\em Weak semistable reduction in
  characteristic 0}, Invent. Math., 139 (2000), pp.~241--273.

\bibitem{AW}
{\sc D.~Abramovich and J.~Wise}, {\em {Birational invariance in logarithmic
  Gromov--Witten theory}}, Comp. Math., 154 (2018), pp.~595--620.

\bibitem{ABPZ}
{\sc H.~Arg{\"u}z, P.~Bousseau, R.~Pandharipande, and D.~Zvonkine}, {\em
  Gromov--witten theory of complete intersections}, arXiv:2109.13323,  (2021).

\bibitem{BNR22}
{\sc L.~Battistella, N.~Nabijou, and D.~Ranganathan}, {\em {Gromov-Witten
  theory via roots and logarithms}}, arXiv:2203.17224,  (2022).

\bibitem{Bou19}
{\sc P.~Bousseau}, {\em Tropical refined curve counting from higher genera and
  lambda classes}, Invent. Math., 215 (2019), pp.~1--79.

\bibitem{Bri97}
{\sc M.~Brion}, {\em Equivariant {Chow} groups for torus actions}, Transform.
  Groups, 2 (1997), pp.~225--267.

\bibitem{CN21}
{\sc F.~Carocci and N.~Nabijou}, {\em Rubber tori in the boundary of expanded
  stable maps}, arXiv:2109.07512,  (2021).

\bibitem{CCUW}
{\sc R.~Cavalieri, M.~Chan, M.~Ulirsch, and J.~Wise}, {\em A moduli stack of
  tropical curves}, {Forum Math. Sigma}, 8 (2020), pp.~1--93.

\bibitem{CJMR2}
{\sc R.~Cavalieri, P.~Johnson, H.~Markwig, and D.~Ranganathan}, {\em Counting
  curves on {H}irzebruch surfaces: tropical geometry and the {F}ock space},
  Math. Proc. Cambridge Philos. Soc., 171 (2021), pp.~165--205.

\bibitem{CMR14b}
{\sc R.~Cavalieri, H.~Markwig, and D.~Ranganathan}, {\em {Tropical
  compactification and the Gromov--Witten theory of $\mathbf P^1$}}, Selecta
  Math., 23 (2017), pp.~1027--1060.

\bibitem{CMR22}
\leavevmode\vrule height 2pt depth -1.6pt width 23pt, {\em Pluricanonical
  cycles and tropical covers}, arXiv:2206.14034,  (2022).

\bibitem{Che10}
{\sc Q.~Chen}, {\em Stable logarithmic maps to {D}eligne-{F}altings pairs {I}},
  Ann. of Math., 180 (2014), pp.~341--392.

\bibitem{CS12}
{\sc Q.~Chen and M.~Satriano}, {\em Chow quotients of toric varieties as moduli
  of stable log maps}, Algebra and Number Theory Journal, 7 (2013),
  pp.~2313--2329.

\bibitem{DSvZ}
{\sc V.~Delecroix, J.~Schmitt, and J.~van Zelm}, {\em admcycles---a {S}age
  package for calculations in the tautological ring of the moduli space of
  stable curves}, J. Softw. Algebra Geom., 11 (2021), pp.~89--112.

\bibitem{FP}
{\sc C.~Faber and R.~Pandharipande}, {\em {Relative maps and tautological
  classes.}}, {J. Eur. Math. Soc. (JEMS)}, 7 (2005), pp.~13--49.

\bibitem{Ful98}
{\sc W.~Fulton}, {\em Intersection theory}, vol.~2 of Ergebnisse der Mathematik
  und ihrer Grenzgebiete. 3. Folge. A Series of Modern Surveys in Mathematics
  [Results in Mathematics and Related Areas. 3rd Series. A Series of Modern
  Surveys in Mathematics], Springer-Verlag, Berlin, second~ed., 1998.

\bibitem{GKM07}
{\sc A.~Gathmann, M.~Kerber, and H.~Markwig}, {\em Tropical fans and the moduli
  space of rational tropical curves}, Compos. Math., 145 (2009), pp.~173--195.

\bibitem{Gra19}
{\sc T.~Graber}, {\em Torus localization for logarithmic stable maps}, in
  Logarithmic enumerative geoimetry and mirror symmetry, Oberwolfach Workshop
  Reports, H.~R. D.~Abramovich, M. van~Garrel, ed., vol.~16, 2019.

\bibitem{GP99}
{\sc T.~Graber and R.~Pandharipande}, {\em Localization of virtual classes},
  Invent. Math., 135 (1999), pp.~487--518.

\bibitem{GV05}
{\sc T.~Graber and R.~Vakil}, {\em Relative virtual localization and vanishing
  of tautological classes on moduli spaces of curves}, Duke Math. J., 130
  (2005), pp.~1--37.

\bibitem{Gro14}
{\sc A.~Gross}, {\em Correspondence theorems via tropicalizations of moduli
  spaces}, Commun. Contemp. Math., 18 (2016), p.~36.
\newblock Id/No 1550043.

\bibitem{Gro15}
\leavevmode\vrule height 2pt depth -1.6pt width 23pt, {\em Intersection theory
  on tropicalizations of toroidal embeddings}, Proc. Lond. Math. Soc. (3), 116
  (2018), pp.~1365--1405.

\bibitem{GS13}
{\sc M.~Gross and B.~Siebert}, {\em {Logarithmic Gromov-Witten invariants}}, J.
  Amer. Math. Soc., 26 (2013), pp.~451--510.

\bibitem{Herr}
{\sc L.~Herr}, {\em {The Log Product Formula}}, arXiv:1908.04936,  (2019).

\bibitem{HMPW}
{\sc L.~Herr, S.~Molcho, R.~Pandharipande, and J.~Wise}, {\em Birational models
  of logarithmic {Gromov-Witten} theory}, In preparation,  (2023).

\bibitem{Hol17}
{\sc D.~Holmes}, {\em Extending the double ramification cycle by resolving the
  {A}bel-{J}acobi map}, J. Inst. Math. Jussieu, 20 (2021), pp.~331--359.

\bibitem{HMPPS}
{\sc D.~Holmes, S.~Molcho, R.~Pandharipande, A.~Pixton, and J.~Schmitt}, {\em
  Logarithmic double ramification cycles}, arXiv:2207.06778,  (2022).

\bibitem{HPS19}
{\sc D.~Holmes, A.~Pixton, and J.~Schmitt}, {\em Multiplicativity of the double
  ramification cycle}, Doc. Math., 24 (2019), pp.~545--562.

\bibitem{HS21}
{\sc D.~Holmes and R.~Schwarz}, {\em Logarithmic intersections of double
  ramification cycles}, Algebr. Geom., 9 (2022), pp.~574--605.

\bibitem{Jan17}
{\sc F.~Janda}, {\em Gromov-{W}itten theory of target curves and the
  tautological ring}, Michigan Math. J., 66 (2017), pp.~683--698.

\bibitem{JPPZ}
{\sc F.~Janda, R.~Pandharipande, A.~Pixton, and D.~Zvonkine}, {\em Double
  ramification cycles on the moduli spaces of curves}, Publ. Math. IH{\'E}S,
  125 (2017), pp.~221--266.

\bibitem{KSZ91}
{\sc M.~Kapranov, B.~Sturmfels, and A.~Zelevinsky}, {\em Quotients of toric
  varieties}, Math. Ann., 290 (1991), pp.~643--655.

\bibitem{KHSUK}
{\sc P.~Kennedy-Hunt, Q.~Shafi, and A.~U.~K. Kumaran}, {\em Tropical refined
  curve counting with descendants}, arXiv:2307.09436,  (2023).

\bibitem{LP09}
{\sc M.~Levine and R.~Pandharipande}, {\em Algebraic cobordism revisited},
  Invent. Math., 176 (2009), pp.~63--130.

\bibitem{MR16}
{\sc T.~Mandel and H.~Ruddat}, {\em Descendant log {G}romov-{W}itten invariants
  for toric varieties and tropical curves}, Trans. Amer. Math. Soc., 373
  (2020), pp.~1109--1152.

\bibitem{MW17}
{\sc S.~Marcus and J.~Wise}, {\em Logarithmic compactification of the
  {A}bel-{J}acobi section}, Proc. Lond. Math. Soc. (3), 121 (2020),
  pp.~1207--1250.

\bibitem{MP06}
{\sc D.~Maulik and R.~Pandharipande}, {\em A topological view of
  {G}romov--{W}itten theory}, Topology, 45 (2006), pp.~887--918.

\bibitem{Mi03}
{\sc G.~Mikhalkin}, {\em Enumerative tropical geometry in {${\mathbb{R}^2}$}},
  J. Amer. Math. Soc, 18 (2005), pp.~313--377.

\bibitem{Mol16}
{\sc S.~Molcho}, {\em Universal stacky semistable reduction}, Israel J. Math.,
  242 (2021), pp.~55--82.

\bibitem{MPS21}
{\sc S.~Molcho, R.~Pandharipande, and J.~Schmitt}, {\em The {H}odge bundle, the
  universal $0$-section, and the log {C}how ring of the moduli space of
  curves}, Compos. Math., 159 (2023), pp.~306--354.

\bibitem{MR21}
{\sc S.~Molcho and D.~Ranganathan}, {\em {A case study of intersections on
  blowups of the moduli of curves}}, arXiv:2106.15194,  (2021).

\bibitem{MW18}
{\sc S.~Molcho and J.~Wise}, {\em The logarithmic {P}icard group and its
  tropicalization}, Compos. Math., 158 (2022), pp.~1477--1562.

\bibitem{NS06}
{\sc T.~Nishinou and B.~Siebert}, {\em Toric degenerations of toric varieties
  and tropical curves}, Duke Math. J., 135 (2006), pp.~1--51.

\bibitem{PPZ15}
{\sc R.~Pandharipande, A.~Pixton, and D.~Zvonkine}, {\em Relations on
  {$\overline{\mathcal M}_{g,n}$} via {$3$}-spin structures}, J. Amer. Math.
  Soc., 28 (2015), pp.~279--309.

\bibitem{Par17}
{\sc B.~Parker}, {\em Three dimensional tropical correspondence formula}, Comm.
  Math. Phys., 353 (2017), pp.~791--819.

\bibitem{R15-Severi}
{\sc D.~Ranganathan}, {\em {Toric Severi degrees are logarithmic Gromov-Witten
  invariants}}, 2015.

\bibitem{R15b}
\leavevmode\vrule height 2pt depth -1.6pt width 23pt, {\em {Skeletons of stable
  maps I: rational curves in toric varieties}}, J. Lond. Math. Soc., 95 (2017),
  pp.~804--832.

\bibitem{R16}
\leavevmode\vrule height 2pt depth -1.6pt width 23pt, {\em {Skeletons of stable
  maps II: superabundant geometries}}, Res. Math. Sci., 4 (2017).

\bibitem{R19b}
\leavevmode\vrule height 2pt depth -1.6pt width 23pt, {\em A note on the cycle
  of curves in a product of pairs}, arXiv:1910.00239,  (2019).

\bibitem{R19}
\leavevmode\vrule height 2pt depth -1.6pt width 23pt, {\em {Logarithmic
  Gromov-Witten theory with expansions}}, Algebraic Geometry, 9 (2022),
  pp.~714--761.

\bibitem{RSW17B}
{\sc D.~Ranganathan, K.~Santos-Parker, and J.~Wise}, {\em Moduli of stable maps
  in genus one and logarithmic geometry, {II}}, Algebra Number Theory, 13
  (2019), pp.~1765--1805.

\bibitem{RU22}
{\sc D.~Ranganathan and J.~Usatine}, {\em {Gromov-Witten theory and invariants
  of matroids}}, Selecta Math., 28 (2022).

\bibitem{RW19}
{\sc D.~Ranganathan and J.~Wise}, {\em Rational curves in the logarithmic
  multiplicative group}, Proc. Amer. Math. Soc., 148 (2020), pp.~103--110.

\bibitem{Sie04}
{\sc B.~Siebert}, {\em Virtual fundamental classes, global normal cones and
  {F}ulton's canonical classes}, in Frobenius manifolds, Aspects Math., E36,
  Friedr. Vieweg, Wiesbaden, 2004, pp.~341--358.

\bibitem{Tho22}
{\sc R.~P. Thomas}, {\em A {K}-theoretic {F}ulton class}, in Facets of
  algebraic geometry. {V}ol. {II}, vol.~473 of London Math. Soc. Lecture Note
  Ser., Cambridge Univ. Press, Cambridge, 2022, pp.~367--379.

\bibitem{Wis11}
{\sc J.~Wise}, {\em Obstruction theories and virtual fundamental classes},
  arXiv:1111.4200,  (2011).

\bibitem{Wu21}
{\sc Y.~Wu}, {\em {Splitting of Gromov-Witten invariants with toric gluing
  strata}}, arXiv:2103.14780,  (2021).

\end{thebibliography}
\end{document}